\documentclass{amsart}
\usepackage{amssymb}
\usepackage{amsfonts}
\usepackage{eufrak}
\usepackage{eucal}
\usepackage{txfonts}
\usepackage{amsmath}
\usepackage{enumerate}
\usepackage{comment}
\usepackage{hyperref}

\newtheorem{thm}{Theorem}[section]

\newtheorem{lem}[thm]{Lemma}

\theoremstyle{definition}

\theoremstyle{remark}
\newtheorem{rem}[thm]{Remark}
\numberwithin{equation}{section}

\makeatletter
\@namedef{subjclassname@2020}{\textup{2020} Mathematics Subject Classification}
\makeatother

\begin{document}
	
	\title[]
	{Yamabe flow and locally conformally flat manifolds with positive pinched Ricci curvature}
	
	\author{Liang Cheng}

	

	\thanks{Liang Cheng's  Research partially supported by National
		Natural Science Foundation of China 12171180
	}
	
	\address{School of Mathematics and Statistics $\&$ Hubei Key Laboratory of Mathematical Sciences, Central  China Normal University, Wuhan, 430079, P.R.China}
	
	\email{chengliang@ccnu.edu.cn }

	\maketitle
	
	\begin{abstract}
		
	By using the Yamabe flow, we prove that if $(M^n,g)$, $n\geq
		3$, is an $n$-dimensional locally conformally flat complete Riemannian manifold  	satisfying $Rc\geq \epsilon Rg>0$,
		where  $\epsilon>0$ is a uniformly
		constant, then $M^n$ must be compact. Our result shows that Hamilton's pinching conjecture also holds for higher dimensional case if we assume additionally the metric is locally conformally flat.

	\end{abstract}

	\section{Introduction}
Bonnet-Myers' theorem is one of the classical theorems in Riemannian
	geometry which states that if $M^n$ is a complete Riemannian
	manifold with its Ricci curvature satisfying $Rc\geq k>0$, then
	$M^n$ must be compact and $diam(M)\leq \frac{\pi}{\sqrt{k}}$.
	It is  interesting to seek
 other conditions on curvatures to get the compactness for
	manifolds. For this direction, there was  an interesting pinching conjecture by Hamilton (\cite{RFV2}, Conjecture 3.39):
\textit{
	If $(M^3,g)$, $n\geq 3$, is a 3-dimensional complete Riemannian
	manifold satisfying $Rc\geq \epsilon
	Rg>0$ for some $\epsilon>0$, where $R$ is the scalar curvature of $(M^3,g)$, then $M^3$ must be compact.}
Chen and Zhu \cite{CZ2000} showed the conjecture is true that
if one assumes additionally that $(M^3,g)$ have the bounded nonnegative sectional curvature. Later, Lott \cite{L} proved the conjecture under the weaker additional assumptions that the sectional
curvature is bounded and has an inverse quadratic lower bound.  Subsequently,
 Deruelle-Schulze-Simon \cite{DSS} proved the conjecture under the additional hypotheses to only require additionally
that the sectional curvature is bounded. 
 Lee and Topping \cite{LP} recently removed the bounded curvature assumption of  Deruelle-Schulze-Simon's result and  solved this conjecture completely. 	Noted that
all the above results obtained by using the Ricci flow.

	In this paper we use Yamabe flow to show that the higher dimensional case of
	Hamilton's conjecture also holds if we assume additionally the metric is
	locally conformally flat.
		\begin{thm}\label{main}
		If $(M^n,g)$, $n\geq 3$, is an $n$-dimensional complete locally  conformally flat Riemannian
		manifold
		satisfying
		\begin{equation}\label{key_pinching}
				Rc\geq \epsilon Rg>0
		\end{equation}
 for some $\epsilon> 0$, then $M^n$ must be compact.
	\end{thm}
	\begin{rem}
		By applying the strong maximum principle to the evolution equation of scalar curvature for Yamabe flow, Theorem \ref{main} is equivalent to
		say if $(M^n,g)$, $n\geq 3$, is an $n$-dimensional complete noncompact
		locally comformally flat manifold satisfying
		$Rc\geq \epsilon Rg$ and $R\geq 0$, then $M^n$ must be flat; see (\ref{scalar}) in Theorem \ref{chow_formula}.
	\end{rem}

	The Yamabe flow was proposed by
R.Hamilton \cite{H89} in the 1980's as a tool for constructing
metrics of constant scalar curvature in a given conformal class.
The Yamabe flow is defined by the
	evolution equation
	\begin{equation}\label{yamabe}
		\left\{
		\begin{array}{ll}
			\frac{\partial g}{\partial t}=-Rg \quad &\text{in}\ M^n\times[0,T),\\
			g(\cdot,0)=g_0 &\text{in}\ M^n,
		\end{array}
		\right.
	\end{equation}
	on an $n$-dimensional complete Riemannian manifold $(M^n,g_0)$, $n\geq 3$, where $g(t)$ is a family
	of Riemannian metrics in the conformal class of $g_0$ and $R$ is the
	scalar curvature of the metric.
	If we write $g(t)=u(t)^{\frac{4}{n-2}}g_0$ with $u$ being a positive smooth
	function on $M$ and change time by a constant scale, then  
	 (\ref{yamabe}) is equivalent to the following heat type
	equation
	\begin{equation}\label{yamabe_flow_u}
		\left\{
		\begin{array}{ll}
			\frac{\partial u^N}{\partial t}=L_{g_0}u, \quad &\text{in}\ M^n\times[0,T),\\
			u(\cdot,0)=1, &\text{in}\ M^n,
		\end{array}
		\right.
	\end{equation}
	where $N=\frac{n+2}{n-2}$, $L_{g_0}u=\Delta_{g_0}u-aR_{g_0}u$ and
	$a=\frac{n-2}{4(n-1)}$.
 The asymptotic 
behaviour of Yamabe flow on compact manifolds was analysed by Chow \cite{chow}, Ye \cite{Y94}, Schwetlick and Struwe \cite{SS}
and Brendle \cite{BS2005}, \cite{BS2007}. 	For the theory of Yamabe flow on noncompact manifolds, we can see \cite{DS}\cite{DKS1}\cite{DKS2}\cite{Ma19}\cite{Ma21} and references therein for more information.

Theorem \ref{main} improves the result obtained Gu \cite{Gu} with an additional hypotheses that $(M^n,g)$ has the bounded non-negative sectional curvature and  by Ma and the author \cite{MC} with an additional hypotheses that $(M^n,g)$ has the bounded curvature. The main progress of this paper is that we prove the
 following theorem:
	\begin{thm}\label{main_key}
	For all $\epsilon>0$, there exists $ Q(\epsilon,n)>0$ such that the following holds. Let $\left(M^n, g_0\right)$ be an $n$-dimensional locally conformally flat complete noncompact manifold such that
	\begin{equation}\label{pinching_cond_intro}
		Rc\left(g_0\right) \geq \epsilon  R\left(g_0\right)g_0> 0
	\end{equation}
	on $M^n$.		
	Then there exists a smoothly locally conformally flat  Yamabe flow solution $g(t)$ defined on $M^n \times[0, +\infty)$ such that $g(0)=g_0$,
	$$
	| Rm(g(x, t))| \leq \frac{Q}{t}
	$$ and
	$$
	Rc(g(x, t)) \geq \epsilon  R(g(x, t))g(x,t)>0
	$$
	for all $(x, t) \in M^n \times[0, +\infty)$, where $Rm$ denotes the Riemannian curvature tensor.
\end{thm}

 Hamilton\cite{H89} proved the local existence of Yamabe flows on compact manifolds without boundary. For the complete noncompact manifolds with bounded scalar curvature, local existence of Yamabe flows  was obtained by An-Ma \cite{AM}  and Chen-Zhu \cite{CZ}. However, there exists no a general existence theorem for the noncompact Yamabe flow without bounded curvature.  Theorem \ref{main_key} shows that the Yamabe flow has a smooth solution such that its curvature becomes bounded as $t>0$ and preserves the pinching condition (\ref{key_pinching}) if its initial metric satisfying the assumptions of Theorem \ref{main}.

We sketch our strategy for  the proof of  Theorem \ref{main}. In order to prove Theorem \ref{main_key}, we 
need to construct the local Yamabe flow for the local ball for which the existence time of local Yamabe flow  is uniform and dependent only  on $\epsilon$ and $n$; see Theorem \ref{main_key}. 
In order to construct such local Yamabe flow, we use the following inductive method which was used by Lee and Topping \cite{LP} for the Ricci flow:
the process
starts by doing a conformal change to the initial metric, making it a complete metric with bounded curvature and
leaving it unchanged on a smaller region, and then run a complete Yamabe flow up to a
short time.
Next we do the conformal change to the metric  again and
repeating the
process. 
 In contrast with the Ricci flow,
 the main problem for us is that the Yamabe flow is NOT a super Ricci flow, i.e. the Yamabe flow do not satisfies
 $
\frac{\partial g}{\partial t}\ge -2Rc
 $
for which has the nice estimates for distance
\begin{equation*}
\left.\left(\frac{\partial}{\partial t}-\Delta_{g(t)}\right) d_{g(t)}\left(x, x_0\right)\right|_{t=t_0} \geq-(n-1)\left(\frac{2}{3} K r_0+\frac{1}{r_0}\right)
\end{equation*}
if $Rc(g\left(t_0\right)) \leq(n-1) K $ on  $B_{g\left(t_0\right)}\left(x_0, r_0\right)\cup B_{g\left(t_0\right)}\left(x, r_0\right)$ and $d_g(x,x_0)>2r_0$; See Lemma 8.3 in \cite{P}. This can be used to construct a good cut-off function which is crucial to prove the pinching condition is preserved for the local Ricci flow; see \cite{LP}.
In order to overcome this problem,
the key observation for us is that if the Yamabe flow satisfying the following the pinching condition which we can assume
it holds by the inductive arguments
\begin{equation}\label{jst}
	Rc(g\left(t_0\right))\ge  \left(\epsilon R(g\left(t_0\right))-\lambda\right)g\left(t_0\right)
\end{equation}
for $t_0\in [0,t_k]$, 
we have the following estimate for  distance
	\begin{equation}\label{good}
		\left.\quad\left(\frac{\partial}{\partial t}-(n-1)\Delta_{g(t)}\right) d_{g(t)}\left(x, x_0\right)\right|_{t=t_0} \geq- \frac{2(n-1)}{\epsilon}\left(K r_0+\frac{1}{r_0} \right)-\frac{\lambda}{\epsilon}d_{g(t_0)}\left(x, x_0\right)
	\end{equation}
if $Rc(g\left(t_0\right)) \leq(n-1) K $ on  $B_{g\left(t_0\right)}\left(x_0, r_0\right)\cup B_{g\left(t_0\right)}\left(x, r_0\right)$ and $d_g(x,x_0)>2r_0$; see Theorem \ref{cut_off_1}.  
Combining with 
 Theorem \ref{cut_off_2}, We can also use these to construct a good cut-off function to show
the pinching condition is preserved under local Yamabe flow on $[0,t_{k+1}]$; see the proof of Theorem \ref{main1}.  

The present paper is organized as follows. In section 2 we recall some basic results and the short-time existence of Yamabe flow on complete manifolds.  In section 3  we do estimates for changing distances under the Yamabe flow, espcially assuming the pinching condition. In section 4 we prove a local $\frac{c}{t}$ estimate under the Yamabe flow satisfying the pinching condition. In section 5 we get an  existence theorem
for the local Yamabe flow. 
In section 6 we give the proof of Theorem \ref{main} and Theorem \ref{main_key}. In the appendix we give the proof of a maximum principle theorem which we use in setion 3.

	\section{preliminaries}
	
	We first recall some basic evolution equations of Yamabe flow obtained by Chow \cite{chow}.
	
	\begin{lem}[Lemma 2.2 and Lemma
		2.4 in \cite{chow}] \label{chow_formula} If $(M^n,g(t))$, $n\geq 3$, is the solution to Yamabe flow (\ref{yamabe}) on
		an $n$-dimensional  Riemannian manifold, then the scalar curvature evolves as
		\begin{equation}\label{scalar}
			R_t=(n-1)\Delta R+R^2.
		 \end{equation}
		Moreover, if $(M^n,g(0))$ is locally  conformally flat, then the Ricci curvature evolves as
		\begin{equation}\label{Rc}
			\partial_t R_{ij}=(n-1)\Delta R_{ij}+\frac{1}{n-2}B_{ij},
		\end{equation}
		where
		$$
		B_{ij}=(n-1)|Rc|^2g_{ij}+nRR_{ij}-n(n-1)R_{ij}^2-R^2g_{ij}.
		$$
	\end{lem}
\begin{rem}\label{shi_esimates}
		We can rewrite the equation
	(\ref{Rc}) for $Rc$ as
	$$
	\partial_t Rc = (n-1)\Delta Rc + Rc\ast Rc,
	$$
	where $Rc\ast Rc$ stands for any linear combination of tensors
	formed by contraction on $R_{ij}\cdot R_{kl}$. Notice that the
	evolution for $Rc$ along the Yamabe Flow for the locally conformally flat case has the same form as the
	evolution for $Rm$ along the Ricci flow. So Shi's techniques
	in \cite{S97} can be applied and we can show that
	all the covariant derivatives of $Rm$ are uniformly bounded on
	$[0,T)$ if $|Rc|$ is bounded on $[0,T)$ under the locally conformally flat Yamabe flow.
\end{rem}

	The following   Hochard's result  that allows us to conformally change  an incomplete Riemannian
	metric at its extremities in order to make it complete and without changing it in the interior.
	
	\begin{thm} [Corollaire IV.1.2 in \cite{Ho}]\label{Ho}
		There exists $\sigma(n)$ such that given a Riemannian manifold $\left(N^n, g\right)$ with $| Rm(g)| \leq \rho^{-2}$ throughout for some $\rho>0$, there exists a complete Riemannian metric $h$ on $N$ such that
		
		(1) $h \equiv g$ on $N_\rho:=\left\{x \in N: B_g(x, \rho) \Subset N\right\}$, and
		
		(2) $| Rm(h)| \leq \sigma \rho^{-2}$ throughout $N$.
	\end{thm}

The short-time existence of smooth solution to Yamabe flow (\ref{yamabe_flow_u}) on noncompact complete manifolds with bounded scalar curvature was obtained by An-Ma \cite{AM}  and Chen-Zhu \cite{CZ}. We shall show that the locally conformally flat solution has bounded sectional curature on some specific time interval if we assume additionally that the sectional curature is bounded at $t=0$.
\begin{thm}\label{short_existence}
	Suppose that $(M^n,g_0)$ is  an $n$-dimensional locally conformally flat complete noncompact smooth manifold with
	$$|Rm|(g_0)\le K$$
	for some $K>1$,
	there exist  positive constants $\beta(n)$ and $\Lambda(n)$ such that
	 the Yamabe flow (\ref{yamabe_flow_u}) has
	a smooth solution $g(t)$ on $M$ for $t\in [0,\frac{\beta}{K} ]$,  with the properties that
	 $g(0)=g_0$ and
	 \begin{equation}\label{yamabe_local_curvature}
 |Rm|(g(t))\le \Lambda K
	 \end{equation}
	for $t\in [0,\frac{\beta}{K} ]$.
\end{thm}
\begin{proof}
We consider  the following Dirichlet problem for a sequence of exhausting bounded smooth domains
\begin{equation}\label{yamabe_flow_u_local1}
	\left\{
	\begin{array}{ll}
		\frac{\partial u^N_m}{\partial t}=L_{g_0}u_m, \quad &\ x\in\Omega_m, \ t>0,\\
		u_m(x,t)>0, \quad &\ x\in\Omega_m, \ t>0,\\
		u_m(x,t)=1, \quad &\ x\in \partial \Omega_m, \ t>0,\\
		u_m(\cdot,0)=1, \quad &\ x\in\Omega_m,\\
	\end{array}
	\right.
\end{equation}
where $\Omega_1\subset\Omega_2\subset\cdots$ such that $\mathop{\cup}\limits^{\infty}_{m=1}\Omega_m=M$. Since  $u_m(x,t)=1$ is bounded on $\partial \Omega_m$, we may assume $u_m(x,t)$ achieves its maximum $\max\limits_{\Omega_m} u_m$ and minimum $\min\limits_{\Omega_m} u_m$ in the interior of $\Omega_m$. By Proposition 2.2 in \cite{CZ}, the Dirichlet problem (\ref{yamabe_flow_u_local1}) has
a unique smooth solution on $[0,+\infty)$.
Then by the maximum principle, we conclude that
\begin{align*}
	\max\limits_{\Omega_m} u_m(t)\leq (1+\frac{n-2}{(n-1)(n+2)}\sup\limits_{M^n}|R_{g_0}|t)^{\frac{n-2}{4}}.
\end{align*}
and
\begin{align*}
	\min\limits_{\Omega_m} u_m(t)\geq (1-\frac{n-2}{(n-1)(n+2)}\sup\limits_{M^n}|R_{g_0}|t)^{\frac{n-2}{4}}.
\end{align*}
Then there  positive constants $\beta(n)$ such that
$$
\frac{1}{2}\le u_m(t)\le \frac{3}{2}
$$
for $t\in [0,\frac{\beta}{K} ]$.
 This shows, for any fixed compact subset $D$ of $M^n$, the equation (\ref{yamabe_flow_u_local1}) is uniformly parabolic on $D \times[0, \frac{\beta}{K}]$. Thus we can get the uniform $C^\alpha$ estimates from Krylov-Safonov estimates \cite{KS} and then the uniform $C^{k, \alpha}$ estimates on $D \times[0, \frac{\beta}{K}]$ from Schauder theory. Hence by the arbitrariness of $D$, it follows from a standard diagonal argument that the Yamabe flow has a positive smooth solution $u$ on $M^n \times[0,\frac{\beta}{K}]$ with $u(t)$ satisfying
\begin{equation}\label{u_equavelent}
\frac{1}{2}\le u(t)\le \frac{3}{2}
\end{equation}
for $t\in [0,\frac{\beta}{K} ]$.

By the virtue of Shi's methods in \cite{S89}, we next show that (\ref{yamabe_local_curvature}) holds for $g(t)=u(t)^{\frac{4}{n-2}}g_0$. Firstly,  we show that
\begin{equation}\label{u_gradient_est}
\sup\limits_{M^n \times[0,\frac{\beta}{K}]} |\nabla_{g_0}u(t)|\le c(n,K).
\end{equation}
Take any fixed $x_0\in M^n$,
 we consider the ball
$
\widehat{B}\left(0, \pi\left(1 / K\right)^{1 / 4}\right) \subseteq T_{x_0} M
$
of radius $\pi\left(1 /K\right)^{1 / 4}$ in the tangent space. Since $\left|Rm_{g_0} \right|^2 \leq K$,  we can use the non-singular exponential map $\exp^{g_0} _{x_0}$ to pull everything back from $M$ to $\widehat{B}\left(0, \pi\left(1 / K\right)^{1 / 4}\right)$ for which the injectivity radius has a lower bound $\pi\left(1 / K\right)^{1 / 4}$,
and do the analysis on $\widehat{B}\left(0, \gamma_0\right) \subseteq T_{x_0} M$ of radius $\gamma_0=\frac{1}{8}\left(1 / K\right)^{1 / 4}$ under which has the  Poincar\'{e} inequality constant and the Sobolev inequality constants depending on $n$ and $K$.
Thus using (\ref{u_equavelent}) and the same arguments as in the proof of Theorem 6.1 in $\S 6$, Chapter VII of \cite{LSU}  to the equation (\ref{yamabe_flow_u}), we can get 
\begin{equation*}
	\sup\limits_{\widehat{B}\left(0, \gamma_0\right) \times[0,\frac{\beta}{K}]} |\nabla_{g_0}u(t)|\le c(n,K).
\end{equation*}
Since $x_0\in M^n$ is arbitrary, we have
(\ref{u_gradient_est}) holds.

We choose a smooth cut-off function 	$\xi(x)\in [0,1]$ on $M^n$ such that
$\xi(x) \equiv 1$ for $x \in B_{g_0}\left(x_0, \gamma_0\right), 
	\xi(x) \equiv 0$ for  $x \in M \backslash B_{g_0}\left(x_0, \gamma_0+\frac{1}{2}\right)$ and $
	 |\nabla_{g_0} \xi(x)| \leq 8.
$
For $t\in [0,\frac{\beta}{K} ]$, we calculate under the metric $g_0$ that
\begin{align*}
	&\ \ \frac{d}{dt}\int_{M} |\nabla u^N|^2\xi\\
	 =&2 \int_{M} \nabla  u^N \cdot \nabla\left( \Delta u-aRu\right) \cdot  \xi\\
	=&2N \int_{M} u^{N-1}\nabla u  \cdot \nabla\left(\Delta u\right)  \cdot \xi-2Na \int_{M}  u^{N-1}\nabla  u \cdot \nabla\left( Ru\right) \cdot  \xi\\
	=& N \int_{M} u^{N-1} \Delta |\nabla u|^2  \cdot \xi-2N \int_{M} u^{N-1}  |\nabla\nabla u|^2  \xi
	-2N \int_{M} u^{N-1} Rc(\nabla u,\nabla u)  \xi-2Na \int_{M}  u^{N-1}\nabla  u \cdot \nabla\left( Ru\right) \cdot  \xi\\
	=& -2N \int_{M} \nabla u^{N-1} \cdot\nabla u\cdot \nabla\nabla u  \cdot \xi-2N \int_{M}  u^{N-1} \cdot\nabla u\cdot \nabla\nabla u  \cdot \nabla \xi-2N \int_{M} u^{N-1}  |\nabla\nabla u|^2  \xi\\
	&
	-2N \int_{M} u^{N-1} Rc(\nabla u,\nabla u)  \xi+2Na \int_{M} \nabla u^{N-1}\cdot\nabla  u \cdot \left( Ru\right) \cdot  \xi+2Na \int_{M}  u^{N-1}\cdot \nabla\nabla  u \cdot \left( Ru\right) \cdot  \xi\\
	&+2Na \int_{M}  u^{N-1}\cdot \nabla  u \cdot \left( Ru\right) \cdot  \nabla\xi\\
		\le& -\frac{N}{2^{N-1}}\int_{M}   |\nabla\nabla u|^2  \xi +c(n,K, \gamma_0),
\end{align*}
where we used the H\"{o}lder inequality, (\ref{u_equavelent}) and (\ref{u_gradient_est}) in the last inequality. Then by integrating the above estimate, we get
\begin{equation}\label{u_2gradient_est}
\int^{\frac{\beta}{K}}_0	\int_{B_{g_0}\left(x_0, \gamma_0\right)}   |\nabla_{g_0}\nabla_{g_0} u|^2  \le c(n,K, \gamma_0)
\end{equation}

If $g(t)$ is locally conformally flat, hence by Lemma \ref{chow_formula} we have that $Rm(g(t))$ evolves:
\begin{equation}\label{Rm1}
		\partial_t Rm = (n-1)\Delta_{g(t)} Rm + Rm\ast Rm.
\end{equation}
Moreover, (\ref{u_equavelent}), (\ref{u_gradient_est}) and (\ref{u_2gradient_est}) imply that
\begin{equation}\label{Rm2}
	 \left(\frac{1}{2}\right)^{\frac{4}{n-2}}g_0\le g(t)\le  \left(\frac{3}{2}\right)^{\frac{4}{n-2}}g_0,
\end{equation}
$t\in [0,\frac{\beta}{K} ]$,
\begin{equation}\label{Rm3}
\sup\limits_{M^n \times[0,\frac{\beta}{K}]} |\nabla_{g_0}g(t)|\le c(n,K,\gamma_0),
\end{equation}
and 
\begin{equation}\label{Rm4}
	\int^{\frac{\beta}{K}}_0	\int_{B_{g_0}\left(x_0, \gamma_0\right)}   |\nabla_{g_0}\nabla_{g_0} g(t)|^2  \le c(n,K, \gamma_0).
\end{equation}
Then the rest of the proofs are similar to  \cite{S89}. Just notice that the related estimates in \cite{S89} are done for the Ricci-DeTurck flow, which are pulled back from the Ricci flow by the diffeomorphisms generating by some vector fields $V(t)$. Since the $Rm(g(t))$ evolves as the same way as the Ricci flow, combining with (\ref{Rm2})(\ref{Rm3}) (\ref{Rm4}), we can just use the similar arguments of \cite{S89} to get (\ref{yamabe_local_curvature}). Indeed, we can exactly use the arguments of 
Lemma 6.3-Lemma 6.5 and Theorem 6.6 in \cite{S89} and taking $V(t)\equiv 0$ in there. So we omit the  details
here and leave them to the readers.
\end{proof}

	\section{  Estimates for changing distances under the Yamabe flow}

In this section we do estimates for changing distances under the Yamabe flow, espcially assuming the pinching condition. These estimates will be crucial to constuct suitable cut-off functions in the proof of Theorem \ref{main1}.  We first need following estimates for the
laplacian of the distances.
\begin{lem}
	Let $\left(M^n, g\right)$ be a smooth Riemannian manifold.
	Given any $r_0>0$, $x_0 \in M$ and $x \in M-\left\{x_0\right\}$ with $D\doteq d(x,x_0)\ge2r_0$ and $B(x_0,D)\Subset M$,
	if
	$Rc\le (n-1)K$ on $B(x_0,r_0)\cup B(x,r_0)$, then we have for any $b\ge1$
	\begin{equation}\label{distance_est_1}
		\quad \Delta d\left(x_0, x\right)
		\leq  -\int_0^{s_0} bRc\left(\gamma^{\prime}(s), \gamma^{\prime}(s)\right) d s+\frac{(n-1)\left(b+(1-\sqrt{b})^2\right)}{r_0}+(2b-1)(n-1) K r_0,
	\end{equation}
	in the barrier sense, and
	if  $Rc\le (n-1)K$ on $B(x_0,r_0)$ and  $Rc\ge -(n-1)K$ on $ B(x,r_0)$, then we have for any $0\le b< 1$
	\begin{equation}\label{distance_est_2}
		\quad \Delta d\left(x_0, x\right)
		\leq  -\int_0^{s_0} bRc\left(\gamma^{\prime}(s), \gamma^{\prime}(s)\right) d s+\frac{(n-1)\left(b+(1-\sqrt{b})^2\right)}{r_0}+(n-1) K r_0,
	\end{equation}
	in the barrier sense, where $\gamma(s)$ is the unit speed minimal geodesic from $x_0$ and $x$.
	
\end{lem}
\begin{proof}
	Given any $x_0 \in M$ and $x \in M-\left\{x_0\right\}$, we have
	\begin{equation}\label{second_var}
		\quad \Delta d\left(x_0, x\right) \leq \int_0^{s_0}\left((n-1)\left(\zeta^{\prime}(s)\right)^2-\zeta^2 Rc\left(\gamma^{\prime}(s), \gamma^{\prime}(s)\right)\right) d s
	\end{equation}
	for any unit speed minimal geodesic $\gamma:\left[0, s_0\right] \rightarrow M$ joining $x_0$ to $x$ and any continuous piecewise $C^{\infty}$ function $\zeta:\left[0, s_0\right] \rightarrow[0,1]$ satisfying $\zeta(0)=0$ and $\zeta\left(s_0\right)=1$; see Theorem 18.6 \cite{RFV3}.	
	Taking
	$
	s_0=d\left(x, x_0\right)>2r_0,
	$
	we may choose
	$$
	\zeta(s)= \begin{cases}\frac{\sqrt{b}s}{r_0}, & \text { if } 0 \leq s \leq r_0, \\ \sqrt{b}, & \text { if } r_0< s \leq s_0-r_0,
		\\ 1+\frac{1-\sqrt{b}}{r_0}(s-s_0), & \text { if } s_0-r_0< s \leq s_0.
	\end{cases}
	$$
	in (\ref{second_var}), so that
	$$
	\begin{aligned}
		\Delta d\left(x_0, x\right) \leq & \int_0^{r_0}\left(\frac{(n-1)b}{r_0^2}-\frac{b s^2}{r_0^2} Rc\left(\gamma^{\prime}(s), \gamma^{\prime}(s)\right)\right) d s \\
		& +\int_{s_0-r_0}^{s_0}\left(\frac{(n-1)(1-\sqrt{b})^2}{r_0^2}-\left(1+\frac{1-\sqrt{b}}{r_0}(s-s_0)\right)^2 Rc\left(\gamma^{\prime}(s), \gamma^{\prime}(s)\right)\right) d s \\
		& -\int_{r_0}^{s_0-r_0} bRc\left(\gamma^{\prime}(s), \gamma^{\prime}(s)\right) d s .
	\end{aligned}
	$$
	We simplify this as
	\begin{align*}
		\Delta & d\left(x_0, x\right) \\
		\leq & -\int_0^{s_0} bRc\left(\gamma^{\prime}(s), \gamma^{\prime}(s)\right) d s\\	& +\int_0^{r_0}\left(\frac{(n-1)b}{r_0^2}+b\left(1-\frac{ s^2}{r_0^2}\right)Rc\left(\gamma^{\prime}(s), \gamma^{\prime}(s)\right)\right) d s \\
		&+\int_{s_0-r_0}^{s_0}\left(\frac{(n-1)(1-\sqrt{b})^2}{r_0^2}+\left(b-\left(1+\frac{1-\sqrt{b}}{r_0}(s-s_0)\right)^2\right) Rc\left(\gamma^{\prime}(s), \gamma^{\prime}(s)\right)\right) d s.
	\end{align*}
	Therefore, in the barrier sense, if $b\ge1$ we have
	$$
	\Delta d\left(x_0, x\right)
	\leq  -\int_0^{s_0} bRc\left(\gamma^{\prime}(s), \gamma^{\prime}(s)\right) d s+\frac{(n-1)\left(b+(1-\sqrt{b})^2\right)}{r_0}+(2b-1)(n-1) K r_0, 		
	$$
	since $Rc\leq  (n-1) K$ along $\left.\gamma\right|_{\left[0, r_0\right]} \subset B\left(x_0, r_0\right) $, $\left.\gamma\right|_{\left[s_0-r_0, s_0\right]} \subset B\left(x, r_0\right) $ and $0< b-\left(1+\frac{1-\sqrt{b}}{r_0}(s-s_0)\right)^2\le b-1$ for  $s\in[s_0-r_0,s_0]$.
	If $0\le b\le 1$, we have	
	$$
	\Delta d\left(x_0, x\right)
	\leq  -\int_0^{s_0} bRc\left(\gamma^{\prime}(s), \gamma^{\prime}(s)\right) d s+\frac{(n-1)\left(b+(1-\sqrt{b})^2\right)}{r_0}+ (n-1)K r_0, 		
	$$
	since $Rc\leq  (n-1) K$ along $\left.\gamma\right|_{\left[0, r_0\right]} \subset B\left(x_0, r_0\right) $, $Rc\ge  -(n-1) K$ along $\left.\gamma\right|_{\left[0, r_0\right]} \subset B\left(x_0, r_0\right)$ $\left.\gamma\right|_{\left[s_0-r_0, s_0\right]} \subset B\left(x, r_0\right) $  and $b-1\le b-\left(1+\frac{1-\sqrt{b}}{r_0}(s-s_0)\right)^2\le 0$ for $s\in[s_0-r_0,s_0]$.
	
\end{proof}

Next we do the estimates for heat operator of the distance  under the Yamabe flow by assuming the pinching condition.
\begin{lem}\label{cut_off_1}
	Let $\left(M^n, g(t)\right), t \in$ $[0, T)$, be a solution to the Yamabe flow.
	Given any $r_0>0$, $x_0 \in M$ and $x \in M-\left\{x_0\right\}$ with $\left(x_0, t_0\right) \in M \times$ $[0, T)$, $D\doteq d(x,x_0)$ and $B_{g\left(t_0\right)}(x_0,D)\Subset M$,
	if
	$$
	Rc(g\left(t_0\right)) \leq(n-1) K \quad \text {on }  B_{g\left(t_0\right)}\left(x_0, r_0\right)\cup B_{g\left(t_0\right)}\left(x, r_0\right),
	$$
	where $K \geq 0$ and
	\begin{equation}\label{pinching_1}
		Rc(g\left(t_0\right))\ge  \left(\epsilon R(g\left(t_0\right))-\lambda\right)g\left(t_0\right)
	\end{equation}
	along all minimal geodesics from $x_0$ and $x$ where $\epsilon\le \frac{1}{n-1}$,
	then we have
	\begin{equation}\label{distance_est_7}
		\left.\frac{\partial}{\partial t}\right|_{t=t_0} d_{g(t)}\left(x_0, x\right) \geq-\frac{2(n-1)}{\epsilon}\left(K r_0+\frac{1}{r_0}\right) -\frac{\lambda}{\epsilon}d_{g(t_0)}\left(x, x_0\right),
	\end{equation}
	in the barrier sense
	and if $d_{g(t_0)}(x_0,x)> 2r_0$, then we have
	\begin{equation}\label{distance_est_2}
		\left.\quad\left(\frac{\partial}{\partial t}-(n-1)\Delta_{g(t)}\right) d_{g(t)}\left(x, x_0\right)\right|_{t=t_0} \geq- \frac{2(n-1)}{\epsilon}\left(K r_0+\frac{1}{r_0} \right)-\frac{\lambda}{\epsilon}d_{g(t_0)}\left(x, x_0\right)
	\end{equation}
	in the barrier sense.
\end{lem}
\begin{proof}
	Firstly, we get from (\ref{pinching_1}) and Lemma 18.4 in \cite{RFV3} that there exists a unit speed minimal geodesic $\gamma $ from $x_0$ to $x$ such that
	\begin{equation}\label{est_1}
		\left.\frac{\partial}{\partial t}\right|_{t=t_0} d_{g(t)}\left(x, x_0\right) \geq-\int_\gamma  R d s \ge -\frac{1}{\epsilon}\int_\gamma Rc(\dot{\gamma}(s),\dot{\gamma}(s)) d s -\frac{\lambda}{\epsilon}d_{g(t_0)}\left(x, x_0\right).
	\end{equation}
	Taking $b=\frac{1}{(n-1)\epsilon}\ge 1$ in (\ref{distance_est_1}), we have
	\begin{equation}\label{est_2}
		(n-1)\Delta_{g\left(t_0\right)} d_{g\left(t_0\right)}\left(x_0, x\right)
		\leq  -\frac{1}{\epsilon}\int_{\gamma} Rc\left(\gamma^{\prime}(s), \gamma^{\prime}(s)\right) d s+\frac{2(n-1)}{\epsilon}\left( K r_0+\frac{1}{r_0}\right).		
	\end{equation}
	Then (\ref{distance_est_2}) follows from (\ref{est_1}) and (\ref{est_2}). Noted that
	(\ref{distance_est_7}) just follows from the idea of the estimates for distance by Hamilton \cite{H2}. By the second variation of arc length, we have
	\begin{equation}\label{est_9}
		\int_{0}^{s_{0}} \phi^{2} R c(X, X) d s \leq(n-1) \int_{0}^{s_{0}}|\dot{\phi}(s)|^{2} d s
	\end{equation}
	for every nonnegative function $\phi(s)$ defined on the interval $\left[0, s_{0}\right]$. If $d_{g(t_0)}(x_0,x)> 2r_0$,  we can choose $\phi(s)$ by
	$$
	\phi(s)= \begin{cases}s, & s \in[0,1]; \\ 1, & s \in\left[1, s_{0}-k_0\right]; \\ \frac{1}{k_0}(s_{0}-s), & s \in\left[s_{0}-k_0, s_{0}\right].\end{cases}
	$$
	Then (\ref{distance_est_7})  follows from (\ref{est_1}) and (\ref{est_9})  if $d_{g(t_0)}(x_0,x)>2r_0$.
	If $d_{g(t_0)}(x_0,x)\le 2r_0$,  then
	\begin{equation*}
		\left.\frac{\partial}{\partial t}\right|_{t=t_0} d_{g(t)}\left(x, x_0\right) \ge -\frac{1}{\epsilon}\int_\gamma Rc(\dot{\gamma}(s),\dot{\gamma}(s)) d s -\frac{\lambda}{\epsilon}d_{g(t_0)}\left(x, x_0\right)
		\ge -\frac{2Kr_0}{\epsilon} -\frac{\lambda}{\epsilon}d_{g(t_0)}\left(x, x_0\right).
	\end{equation*}
	Then (\ref{distance_est_7}) holds for any case.
\end{proof}

We also have the following lemma, which is just a corollary of Lemma \ref{cut_off_1}.

\begin{lem}\label{shrinking}
	There exist constants $\gamma=\gamma(n,\epsilon) $ and $\zeta(n,\epsilon)$ depending only on $n$ and $\epsilon$ such that the following is true. Suppose $\left(N^n, g(t)\right)$ is a Yamabe flow for $t \in[0, T]$ with $g(0)=g_0$ and $x_0 \in N$ with $B_{g_0}\left(x_0, r\right) \Subset N$ for some $r>0$, which satisfies
	$$
	Rc(g(t)) \leq \frac{Q}{t}
	$$
	and
	\begin{equation}\label{pinching_2}
		Rc(g\left(t\right))\ge  \left(\epsilon R(g\left(t\right))-1\right)g\left(t\right)
	\end{equation}
	on $B_{g_0}\left(x_0, r\right)$ for each $t \in(0, T]$ with $\epsilon\le \frac{1}{n-1}$, we have
	$$d_{g_0}(x,x_0)\le e^{\zeta T}\left(d_{g(T)}(x,x_0)+\gamma\sqrt{QT}\right)$$ on $B_{g_0}\left(x_0, r\right)$
	and hence
	$$
	B_{g(T)}\left(x_0,e^{-\zeta T}r-\gamma \sqrt{Q T} \right) \subset B_{g_0}\left(x_0, r\right).
	$$	
\end{lem}
\begin{proof}
	The Lemma follows from (\ref{distance_est_7}) by choosing $r_0=\sqrt{t}$.

\end{proof}

Finally,  we have the following the estimates for heat operator of the distance  under the Yamabe flow by assuming Ricci curvature is bounded.	
\begin{lem}\label{cut_off_2}
	Let $\left(M^n, g(t)\right), t \in$ $[0, T)$, be a solution to the Yamabe flow.
	Given any $r_0>0$, $x_0 \in M$ and $x \in M-\left\{x_0\right\}$ with $\left(x_0, t_0\right) \in M \times$ $[0, T)$, $D\doteq d(x,x_0)$, $B_{g\left(t_0\right)}(x_0,D)\Subset M$ and $d_{g(t_0)}(x_0,x)> 2r_0$, if
	$$
	|Rc(g\left(t_0\right))| \leq(n-1) K \quad \text {on }  M^n,
	$$
	where $K \geq 0$ ,
	then we have
	\begin{equation}\label{distance_est_3}
		\left.\quad\left(\frac{\partial}{\partial t}-(n-1)\Delta_{g(t)}\right) d_{g(t)}\left(x, x_0\right)\right|_{t=t_0} \geq -(n-1)\left(\frac{1}{r_0}+ K r_0\right)-n(n-1)Kd_{g(t_0)}\left(x, x_0\right)
	\end{equation}
	in the barrier sense.
\end{lem}
\begin{proof}
	Firstly we have for a minimal geodesic $\gamma $ such that
	\begin{equation}\label{est_3}
		\left.\frac{\partial}{\partial t}\right|_{t=t_0} d_{g(t)}\left(x, x_0\right) \geq-\int_\gamma  Rd s .
	\end{equation}
	Taking $b= 1$ in (\ref{distance_est_1}), we have
	\begin{equation}\label{est_4}
		\Delta_{g\left(t_0\right)} d_{g\left(t_0\right)}\left(x_0, x\right)
		\leq  -\int_0^{s_0} Rc\left(\gamma^{\prime}(s), \gamma^{\prime}(s)\right) d s+(n-1)\left(\frac{1}{r_0}+ K r_0\right).		
	\end{equation}
	Then (\ref{distance_est_3}) follows from (\ref{est_3}) and (\ref{est_4}).
\end{proof}

	\section{A local $\frac{c}{t}$ estimate under the pinching condition }
	
	In this section we get $\frac{c}{t}$ estimate of curvature if the pinching condition is preserved under the local Yamabe flow. Firstly, we need the following lemma which obtained in \cite{MC} by Ma and the author. We present the proof here for the sake of completeness.

	\begin{lem}\cite{MC}\label{pinching_ineq}
		Suppose that $(M^n,g(t))$, for $t\in[0,T]$ $n\geq 3$, is an $n$-dimensional  locally  conformally Yamabe flow
		satisfying
		\begin{equation}\label{pinching_condition_2}
			Rc(g(t))\geq \epsilon R(g(t))g(t)>0,
		\end{equation}
		we have
\begin{eqnarray}\label{pinching_const}
	(\partial_t-(n-1)\Delta) f^{\frac{1}{\delta}}\le -3 f^{\frac{2}{\delta}},
\end{eqnarray}
where $f=\frac{|Rc|^2-\frac{1}{n}R^2}{R^{2-\delta}}$, $\delta=\frac{n\epsilon}{3}$.
\end{lem}
\begin{proof}
By (\ref{Rc}) and $|Rc|^2=g^{ik}g^{jl}R_{ij}R_{kl}$, we have
\begin{eqnarray*}
	\partial_t|Rc|^2&=&2g^{ik}g^{jl}(\partial_tR_{ij})R_{kl}+2Rg^{ik}g^{jl}R_{ij}R_{kl}\\
	&=&(n-1)\Delta |Rc|^2-2(n-1)|\nabla
	Rc|^2+6\frac{n-1}{n-2}R|Rc|^2\\
	&&-\frac{2}{n-2}R^3-\frac{2n(n-1)}{n-2}tr(Rc^3).
\end{eqnarray*}
From (\ref{scalar}), we get
\begin{eqnarray*}
	\partial_t R^2=(n-1)\Delta R^2-2(n-1)|\nabla R|^2+2R^3.
\end{eqnarray*}
Hence
\begin{eqnarray*}
	\partial_t(|Rc|^2-\frac{1}{n}R^2)
	&=&(n-1)\Delta (|Rc|^2-\frac{1}{n}R^2)-2(n-1)(|\nabla
	Rc|^2-\frac{1}{n}|\nabla R|^2)\\
	&&+6\frac{n-1}{n-2}R|Rc|^2-(\frac{2}{n-2}+\frac{2}{n})R^3-\frac{2n(n-1)}{n-2}tr(Rc^3).
\end{eqnarray*}
Now we denote $\partial_t-(n-1)\Delta$ by $\Box$ and we have
\begin{eqnarray*}
	\Box f
	&=&\frac{\Box(|Rc|^2-\frac{1}{n}R^2)}{R^{2-\delta}}
	-(2-\delta)\frac{|Rc|^2-\frac{1}{n}R^2}{R^{3-\delta}}\Box R\\
	&&-(2-\delta)(3-\delta)(n-1)\frac{|Rc|^2-\frac{1}{n}R^2}{R^{4-\delta}}|\nabla
	R|^2\\
	&&+\frac{2(2-\delta)(n-1)}{R^{3-\delta}}\langle\nabla
	R, \nabla (|Rc|^2-\frac{1}{n}R^2)\rangle \\
	&\doteq&A+\frac{1}{R^{2-\delta}}(\delta(|Rc|^2-\frac{1}{n}R^2)R-J),
\end{eqnarray*}
where
\begin{eqnarray*}
	A&\doteq&-\frac{2(n-1)}{R^{2-\delta}}(|\nabla
	Rc|^2-\frac{1}{n}|\nabla
	R|^2)-(2-\delta)(3-\delta)(n-1)\frac{|Rc|^2-\frac{1}{n}R^2}{R^{4-\delta}}|\nabla
	R|^2\\
	&&+\frac{2(2-\delta)(n-1)}{R^{3-\delta}}\langle\nabla R,
	\nabla(|Rc|^2-\frac{1}{n}R^2)\rangle
\end{eqnarray*} 
and
$J=\frac{2}{n-2}(n(n-1)tr(Rc^3)+R^3-(2n-1)R|Rc|^2)$.
Since
$$
\nabla\frac{|Rc|^2-\frac{1}{n}R^2}{R^{2-\delta}}=\frac{\nabla(|Rc|^2-\frac{1}{n}R^2)}{R^{2-\delta}}
-(2-\delta)\frac{|Rc|^2-\frac{1}{n}R^2}{R^{3-\delta}}\nabla R,
$$
we get
\begin{align*}
	\frac{A}{n-1} &=-\frac{2 }{R^{2-\delta}}(|\nabla
	Rc|^2-\frac{1}{n}|\nabla R|^2)-(2-\delta)(1+\delta)
	\frac{|Rc|^2-\frac{1}{n}R^2}{R^{4-\delta}}|\nabla
	R|^2\\
	&+\frac{2(1-\delta)}{R} \langle\nabla R,
	\nabla\frac{|Rc|^2-\frac{1}{n}R^2}{R^{2-\delta}}\rangle+\frac{2}{R^{3-\delta}}
	\langle\nabla R, \nabla(|Rc|^2-\frac{1}{n}R^2)\rangle.
\end{align*}
Since
$$
-\frac{2}{R^{2-\delta}}|\nabla
Rc|^2-2\frac{|Rc|^2}{R^{4-\delta}}|\nabla
R|^2+\frac{2}{R^{3-\delta}}\langle\nabla R, \nabla
|Rc|^2\rangle=-\frac{2}{R^{4-\delta}}|R\nabla Rc-\nabla R  Rc |^2,
$$
so we have
\begin{align}\label{A}
	\frac{A}{n-1} &=\frac{2(1-\delta)
	}{R}\langle\nabla\frac{|Rc|^2-\frac{1}{n}R^2}{R^{2-\delta}},\nabla R\rangle
	-\frac{2 }{R^{4-\delta}}|R\nabla Rc-\nabla R  Rc |^2\nonumber\\
	&-\frac{(1-\delta)\delta
	}{R^{4-\delta}}(|Rc|^2-\frac{1}{n}R^2)|\nabla R|^2.
\end{align}
Then we conclude that
\begin{eqnarray}
\Box f&=&\frac{2(1-\delta)(n-1)}{R}\langle\nabla
f,\nabla R\rangle-\frac{2(n-1)}{R^{4-\delta}}|R\nabla Rc-\nabla R  Rc |^2\nonumber\\
&&-\frac{(1-\delta)\delta(n-1)}{R^{4-\delta}}(|Rc|^2-\frac{1}{n}R^2)|\nabla
R|^2\nonumber\\
&&+\frac{1}{R^{2-\delta}}(\delta
R(|Rc|^2-\frac{1}{n}R^2)-J)\nonumber
\\
 &\le&\frac{2(1-\delta)(n-1)}{R}\langle\nabla
 f,\nabla R\rangle-(1-\delta)\delta(n-1)\frac{|\nabla
 	R|^2}{R^{2}}f\nonumber\\
 &&+\frac{1}{R^{2-\delta}}(\delta
 R(|Rc|^2-\frac{1}{n}R^2)-J)\nonumber\\
  &\le&\frac{(n-1)(1-\delta)}{\delta}\frac{|\nabla f|^2}{f}+\frac{1}{R^{2-\delta}}(\delta
  R(|Rc|^2-\frac{1}{n}R^2)-J),\nonumber
\end{eqnarray}
where we use $\frac{1}{R}\langle\nabla f,\nabla R\rangle \le \frac{1}{2\delta}\frac{|\nabla f|^2}{f}+\frac{\delta}{2}\frac{|\nabla R|^2}{R^2}f$.
Since
 $J\ge \frac{4}{3}n\epsilon  R (|Rc|^2-\frac{1}{n}R^2)$
by (\ref{pinching_condition_2}) and Lemma 5.3 in \cite{MC},
we get
\begin{align}\label{key}
\Box f \le \frac{(n-1)(1-\delta)}{\delta}\frac{|\nabla f|^2}{f}-n\epsilon f^{1+\frac{1}{\delta}},
\end{align}
where we use $f\le R^{\delta}$ in the above inequality.
Then (\ref{pinching_const}) follows from (\ref{key}) directly.
\end{proof}

	\begin{lem}\label{key_lemma}
	Let $M^n$ be a non-compact  manifold and $g(t), t \in$ $[0, T]$ is a smooth solution to the locally conformally flat Yamabe flow on $M$ such that for some $x_0 \in M$,  $B_{g(t)}\left(x_0, 1\right) \Subset M$ for $t \in[0, T]$ satisfying
	\begin{equation}\label{pinching_condition_1}
	\quad Rc(g(t)) \geq \left(\epsilon  R(g(t))-1\right)g(t)
	\end{equation}
			 on $ B_{g(t)}\left(x_0, 1\right), t \in[0, T].$
			Then there exist $\alpha(\epsilon,n)$, $S(\epsilon,n)>0$ such that for $t \in\left(0, \min\{T, S\}\right]$ we have
		$$
		\left| Rm\left(g(x_0, t)\right)\right| \leq \frac{\alpha}{t}.
		$$
	\end{lem}

\begin{proof}
We adapt the idea of \cite{LP} and argue by the contradition.
Suppose the conclusion is false, then there exist a sequence of $n$-dimensional locally conformally flat  manifolds $M_k$, Yamabe flows $g_k(t), t \in\left[0, T_k\right]$ on $M_k$ and points $x_k \in M_k$ satisfying

(i) $B_{g_k(t)}\left(x_k, 1\right) \Subset M_k$ for $t \in\left[0, t_k\right]$ and $t_k \in\left(0, T_k\right]$;

(ii) $Rc(g_k(t)) \geq \left(\epsilon  R(g_k(t))-1\right)g_k(t)$ on $B_{g_k(t)}\left(x_k, 1\right)$ for $t \in\left[0, t_k\right]$;

(iii) $\left| Rm\left(g_k(x_k,t)\right)\right|<\alpha_k t^{-1}$ for $t \in\left(0, t_k\right)$;

(iv) $\left| Rm\left(g_k\left(x_k,t_k\right)\right)\right|=\alpha_k t_k^{-1}$;

(v) $\alpha_k t_k \rightarrow 0$.

By (iv) and the fact that $a_k t_k \rightarrow 0$, a point-picking argument of Theorem 5.1 in \cite{T2} implies that for $k$ large enough, we can find $\beta>0, \tilde{t}_k \in\left(0, t_k\right]$ and $\tilde{x}_k \in B_{g_k\left(\tilde{t}_k\right)}\left(x_k, \frac{3}{4}-\frac{1}{2} \beta \sqrt{\alpha_k \tilde{t}_k} \right)$ such that
 $$
 \left| Rm\left(g_k(x, t)\right)\right| \leq 4
 | Rm\left(g_k\left(\tilde{x}_k, \tilde{t}_k\right)\right) |=4 Q_k
 $$
whenever $d_{g_k\left(\tilde{t}_k\right)}\left(x, \tilde{x}_k\right)<\frac{1}{8} \beta \alpha_k Q_k^{-1 / 2}$ and $\tilde{t}_k-\frac{1}{8} \alpha_k Q_k^{-1} \leq t \leq \tilde{t}_k$ where $\tilde{t}_k Q_k \geq\alpha_k \rightarrow+\infty$.

We consider the  rescaling sequence  $\tilde{g}_k(t)=Q_k g_k\left(\tilde{t}_k+Q_k^{-1} t\right)$ for $t \in\left[-\frac{1}{8} \alpha_k, 0\right]$ such that

(a) $\left| Rm(\tilde{g}_k(\tilde{x}_k,0))\right|=1$;

(b) $\left| Rm(\tilde{g}_k(t))\right| \leq 4$ on $B_{\tilde{g}_k(0)}\left(\tilde{x}_k, \frac{1}{8} \beta \alpha_k\right) \times\left[-\frac{1}{8} \alpha_k, 0\right]$, and

(c)  $Rc(\tilde{g}_k(t)) \geq \left(\epsilon  R(\tilde{g}_k(t))-Q_k^{-1}\right)\tilde{g}_k(t)$  for $t \in\left(0, t_k\right)$ on $B_{\tilde{g}_k(0)}\left(\tilde{x}_k, \frac{1}{8} \beta \alpha_k\right)$,

By (b), we have a universal $\rho>0$ so that the conjugate radius on $\left(B_{\tilde{g}_k(0)}\left(\tilde{x}_k, \frac{1}{8} \beta \alpha_k-\rho\right), \tilde{g}_k(t)\right) $ with $t\in\left[-\frac{1}{8} \alpha_k, 0\right]$ is always larger than $\rho$. Notice that we have no uniform lower bound on the injectivity radius of $\tilde{g}_k(0)$.
 Therefore we can lift $\left(B_{\tilde{g}_k(0)}\left(\tilde{x}_k, \rho\right), \tilde{g}_k(t)\right)$ to $\left(B_{\text {euc }}(o_{\tilde{x}_k},\rho),\phi_{\tilde{x}_k}^*\tilde{g}_k(t)\right)$ by the exponential map of $\tilde{g}_k(0)$ at $\tilde{x}_k$,
where $\phi_k=exp_{\tilde{x}_k,\tilde{g}_k(0)}$ and $B_{\text {euc }}(o_{\tilde{x}_k},\rho)$ is the Euclidean ball on the tangent space at $\tilde{x}_k$. Then we deduce that $\left(B_{\text {euc }}(o_{\tilde{x}_k},\rho),\phi_{\tilde{x}_k}^* \tilde{g}_k(t)\right)$
subconverges to $(B_{\text {euc }}(o_{x_{\infty}},\rho), \tilde{g}_{\infty}(t) )$ in  $C^{\infty}$-sense. It follows from (c) and $\epsilon <\frac{1}{n}$ that
$$
 R\left(\tilde{g}_{\infty}(t)\right)  \geq 0 ;
$$
and
$$
 Rc\left(\tilde{g}_{\infty}(t)\right)  \ge\epsilon R\left(\tilde{g}_{\infty}(t)\right) \tilde{g}_{\infty}(t),
$$
on $B_{\text {euc }}(o_{x_{\infty}},\rho)\times (-\infty,0]$.
If $R(\tilde{g}_{\infty}(t))=0$ at some point in $B_{\text {euc }}(o_{x_{\infty}},\rho)\times (-\infty,0]$, then we get $R(\tilde{g}_{\infty}(t))\equiv0$ by applying the strong maximum principle to (\ref{scalar}). Since  $\tilde{g}_{\infty}(t)$ is locally conformally flat and $Rc\left(\tilde{g}_{\infty}(t)\right)$ is nonnegative, we conclude that   $ Rm\equiv 0$ on $B_{\text {euc }}(o_{x_{\infty}},\rho)\times (-\infty,0]$ which contradicts to (a).
By Lemma \ref{pinching_ineq}, we get
$$
	(\partial_t-(n-1)\Delta) f_{\infty}^{\frac{1}{\delta}}\le -3 f_{\infty}^{\frac{2}{\delta}}
$$
on $B_{\text {euc }}(o_{x_{\infty}},\rho)\times (-\infty,0]$,
where $f_{\infty}=\frac{|Rc|^2-\frac{1}{n}R^2}{R^{2-\delta}}\left(\tilde{g}_{\infty}(t)\right)$, $\delta=\frac{n\epsilon}{3}$. In this case, we can show the maximum principle still holds and we will give the proof by using the spirit of Omori-Yau maximum principle; see Theorem \ref{maximun_principle} in the appendix. Notice that $f_k=\frac{|Rc|^2-\frac{1}{n}R^2}{R^{2-\delta}}\left(\tilde{g}_{k}(t)\right)\le R^{\delta}\left(\tilde{g}_{k}(t)\right)\le C$ on $B_{\tilde{g}_k(0)}\left(\tilde{x}_k, \frac{1}{8} \beta \alpha_k\right) \times\left[-\frac{1}{8} \alpha_k, 0\right]$  and using Theorem \ref{maximun_principle} to $f_{\infty}$ on $[-T,0]$, we get
$
\sup f_{\infty}^{\frac{1}{\delta}}(0)\le \frac{1}{3T}.
$
Let $T\to +\infty$ and then we have $f_{\infty}(0)\equiv 0$ and $Rc\left(\tilde{g}_{\infty}(0)\right)\equiv c>0$ on $B_{\text {euc }}(o_{x_{\infty}},\rho)$.

Now we can use the same arguments as Lemma 3.3 in \cite{LP} to extend this control to a large region.  For any fixed $r>\rho$ and each $k$, we take a maximal disjoint collection of balls $B_{\tilde{g}_k(0)}\left(y_k^j, \rho\right)$ within $B_{\tilde{g}_k(0)}\left(\tilde{x}_k, r\right)$, indexed by $j$. By volume comparison, the number of points $y_k^j$ is bounded uniformly in $k$.
We can use the previous process of lifting by the exponential map on  $B_{\tilde{g}_k(0)}\left(y_k^j, \rho\right)$ and subconverges to give a limit that is of constant Ricci curvature. By a standard diagonal argument and  because of the overlaps between the covering balls $B_{\tilde{g}_k(0)}\left(y_k^j, \rho\right)$ we deduce that each limit has the same constant sectional curvature. In particular,  we have $\operatorname{Ric}\left(\tilde{g}_k(0)\right)>\frac{c}{2}>0$ on $B_{\tilde{g}_k(0)}\left(\tilde{x}_k, r\right)$ for all $k$ sufficient large.
However, by the Bonnet-Myers theorem, we have any length of minimal geodesics in $B_{\tilde{g}_k(0)}\left(\tilde{x}_k, r\right)$ are less than $\frac{\pi}{\sqrt{\frac{c}{2}}}$  which is impossible when $r$ is sufficient large as $\alpha_k\to \infty$.
\end{proof}

	\section{Existence for the local Yamabe flow }
In this section we prove the following existence theorem
for the local Yamabe flow. Noted that the existence time of local Yamabe flow in Theorem \ref{main1} is uniform and dependent only  on $\epsilon$ and $n$. We can exploit Theorem \ref{main1} to get the existence of Yamabe flow on complete noncompact locally conformally flat manifolds satisfying (\ref{pinching_cond_intro}); see the proof of Theorem \ref{main} in section \ref{proof}.
	
	\begin{thm}\label{main1}
	For all $\frac{1}{n}>\epsilon>0$, there exist $T(\epsilon,n), Q(\epsilon,n)>0$ such that the following holds. Let $\left(M^n, g_0\right)$ be an $n$-dimensional locally conformally flat manifold  (not necessarily complete)  and $p \in M$ such that
	$B_{g_0}(p, r+4) \Subset M$ for some $r \geq 1$ and
	\begin{equation}\label{pinching_cond_intro}
		Rc\left(g_0\right) \geq \epsilon  R\left(g_0\right)g_0\ge 0
	\end{equation}
	on $B_{g_0}(p, r+4)$.		
	Then there exists a smoothly locally conformally flat  Yamabe flow solution $g(t)$ defined on $B_{g_0}(p, r) \times[0, T]$ such that $g(0)=g_0$,
	$$
	| Rm(g(x, t))| \leq \frac{Q}{t}
	$$ and
	$$
	Rc(g(x, t)) \geq \left(\epsilon  R(g(x, t))-1\right)g(x,t)
	$$
	for all $(x, t) \in B_{g_0}(p, r) \times(0, T]$, where $Rm$ denotes the Riemannian curvature tensor.
\end{thm}	
\begin{proof}
 For the given pinching constant $\epsilon$, let
	$\Lambda$, $\beta$ be the constants from Theorem \ref{short_existence}; and
	$\alpha$, $S$ be the constants from Lemma \ref{key_lemma}. Take
	$Q:=\max \left\{\Lambda \beta, \Lambda\left(\alpha+\beta\right), 1\right\}$ and 	$\mu=\sqrt{1+\beta \alpha^{-1}}-1>0$.
	Let $\Gamma\ge 1$ be a positive constant which we will determine later.
		
		Choose $1 \geq \rho_0>0$ sufficiently small so that for all $x \in B_{g_0}(p, r+4)$, we have $\left| Rm\left(g_0\right)\right| \leq \rho_0^{-2}$. By Theorem \ref{Ho}, applied with $N=B_{g_0}(p, r+4)$, we  have a complete metric $h$ on $N$ such that $h\equiv g_0$ on $B_{g_0}(p, r+3) $ and $| Rm(h))|\le \sigma\rho_0^{-2}$ on $N$. It follows from Theorem \ref{short_existence} that we can find a complete smooth solution $h(t)$ to the Yamabe flow on $N\times\left[0, \beta \rho_0^2\right]$ and denote $g(t)=h(t)$  on $B_{g_0}(p, r+3) \times\left[0, \beta \rho_0^2\right]$ with
	initial data $g(0)=g_0$,	
		 $$| Rm(g(t))| \leq \Lambda \rho_0^{-2}$$ 	and
		\begin{equation*}
			Rc(g(t)) \geq \left(\epsilon  R(g(t))-\left(\frac{1}{4} \Gamma \sqrt{Q \beta \rho_0^2}\right)^{-2}\right)g(t)
		\end{equation*}
		on
		$B_{g_0}(p, R+$ $3) \times\left[0, \beta \rho_0^2\right]$.
		
		We now define sequences of times $t_k$ and radii $r_k$ inductively as follows:\\
		(a) $t_1=\beta \rho_0^2, r_1=r+3$;\\
		(b) $t_{k+1}=(1+\mu)^2 t_k=\left(1+\beta \alpha^{-1}\right) t_k$ for $k \geq 1$;\\
		(c) $r_{k+1}=r_k-\Gamma \sqrt{Q t_k}$ for $k \geq 1$.
		
		Let $\mathcal{P}(k)$ be the following statement: there is a smoothly locally conformally flat Yamabe flow solution $g(t)$ defined on $B_{g_0}\left(p, r_k\right) \times\left[0, t_k\right]$ with $g(0)=g_0$ such that
		 $$
		 | Rm(g(t))| \leq \frac{Q}{t},
		 $$
		on $B_{g_0}\left(p, r_k\right) \times\left[0, t_k\right]$
		and
		\begin{equation}\label{key_pinching_1}
			Rc(g(t)) \geq \left(\epsilon  R(g(t))-\left(\frac{1}{4} \Gamma \sqrt{Q t_k}\right)^{-2}\right)g(t)
		\end{equation}
		on
		$B_{g_ 0}\left(p, r_{k}-\frac{1}{2} \Gamma \sqrt{Q t_k}\right) \times\left[0, t_k\right]$.
		
Since  $Q \geq \Lambda \beta$, 	we have proved $\mathcal{P}(1)$ is true. Our goal is to show that $\mathcal{P}(k)$ is true for all $k$ provided $r_k>0$.
		We now perform an inductive argument. Suppose that $\mathcal{P}(k)$ is true, we next show that $\mathcal{P}(k+1)$ is true provided that $r_{k+1}>0$.
		For any $x \in B_{g_0}\left(p, r_{k}-\frac{1}{2} \Gamma \sqrt{Q t_k}\right)$,
		we have
		$
		B_{g_0}\left(x, \frac{1}{4} \Gamma \sqrt{Q t_k}\right) \Subset B_{g_0}\left(p, r_k\right) .
		$
		 Consider the rescaled Yamabe flow $\bar{g}(t)=\lambda_1^{-2} g\left(\lambda_1^2 t\right), t \in\left[0,16 \Gamma^{-2} Q^{-1}\right]$ where $\lambda_1=\frac{1}{4} \Gamma \sqrt{Q t_{k}}$ such that $B_{g_0}\left(x, \frac{1}{4} \Gamma \sqrt{Q t_{k}}\right)=B_{\bar{g}(0)}(x, 1)$
		 	and
		 	\begin{equation*}
		 	Rc(\bar{g}(t)) \geq \left(\epsilon  R(\bar{g}(t))-1\right)\bar{g}(t)
		 	\end{equation*}
		 	on
		 	$B_{\bar{g}(0)}(x, 1) \times\left[0, 16 \Gamma^{-2} Q^{-1}\right]$.
		 	 If we choose $\Gamma$ such that  $S\geq 16 \Gamma^{-2} Q^{-1}$,
 we can apply Lemma \ref{key_lemma} over the whole time interval $\left[0,16 \Gamma^{-2} Q^{-1}\right]$ to get
$$
| Rm(\bar{g}(x,t))| \leq \frac{\alpha}{t},
$$
	for
$t\in \left[0, 16 \Gamma^{-2} Q^{-1}\right]$.
		Rescaling the conclusion of  back to $g(t)$ shows that
	 \begin{equation}\label{equation_1}
| Rm(g(t))| \leq \frac{\alpha}{t},
	 \end{equation}
	on $B_{g_0}\left(x, r_{k}-\frac{1}{2} \Gamma \sqrt{Q t_k}\right) \times\left[0, t_k\right]$.

		Denote $N=B_{g_0}\left(p, r_{k}-\frac{1}{2}  \Gamma \sqrt{Q t_k}\right)$ so that for $h_0=g\left(t_k\right)$, estimate (\ref{equation_1}) gives
		$$
		\sup\limits _{N}\left| Rm\left(h_0\right)\right| \leq \rho^{-2}
		$$
		 where $\rho=\sqrt{t_k \alpha^{-1}}$. Moreover, for $y \in B_{g_0}\left(p, r_{k+1}\right)$, we again consider the rescaled Yamabe flow $\bar{g}(t)=\lambda_1^{-2} g\left(\lambda_1^2 t\right), t \in\left[0,16 \Gamma^{-2} Q^{-1}\right]$ where $\lambda_1=\frac{1}{4} \Gamma \sqrt{Q t_{k}}$ such that $B_{g_0}\left(y, \frac{1}{4} \Gamma \sqrt{Q t_{k}}\right)=B_{\bar{g}(0)}(y, 1)$
		 and
		 \begin{equation*}
		 	Rc(\bar{g}(t)) \geq \left(\epsilon  R(\bar{g}(t))-1\right)\bar{g}(t).
		 \end{equation*}
		 on
		 $B_{\bar{g}(0)}(y, 1) \times\left[0, 16 \Gamma^{-2} Q^{-1}\right]$.
		 By the Lemma \ref{shrinking} and if we choose $\Gamma$ such that $\Gamma \sqrt{Q \alpha} \geq 16$, $\gamma \Gamma^{-1}\le \frac{1}{32}$ and $e^{-16\zeta \Gamma^{-2}Q^{-1}}\ge \frac{3}{4}$, we conculde that
		 $$ B_{\bar{g}\left(16\Gamma^{-2}Q^{-1}\right)}\left(y, \frac{\rho}{\frac{1}{4}\Gamma \sqrt{Qt_k}}\right) \subset B_{\bar{g}\left(16\Gamma^{-2}Q^{-1}\right)}\left(y, \frac{1}{4}\right) \subset B_{\bar{g}_0}\left(y, \frac{1}{2}\right) \Subset N,$$
		 and rescaling the conclusion of  back to $g(t)$ shows that
		 $$B_{g\left(t_k\right)}(y, \rho) \subset B_{g\left(t_k\right)}\left(y, \frac{1}{16} \Gamma \sqrt{Q t_k}\right) \subset B_{g_0}\left(y, \frac{1}{8} \Gamma \sqrt{Q t_k}\right) \Subset N.$$
		 This shows that $B_{g_0}\left(p, r_{k+1}\right) \subset N_\rho$, where $N_\rho=\left\{x \in N: B_{g\left(t_k\right)}(x, \rho) \Subset N\right\}$. Hence, we may apply Theorem \ref{Ho} and Theorem \ref{short_existence} again to find a Yamabe flow $g(t)$ on $B_{g_0}\left(p, r_{k+1}\right) \times\left[t_k, t_k+\beta \rho^2\right]$, extending the existing $g(t)$ on this smaller ball, with
		\begin{equation}\label{kand1}
			| Rm(g(t))| \leq \Lambda \rho^{-2}= \frac{\Lambda \alpha}{t_k} \leq \frac{Q}{t}
		\end{equation}
		since $\Lambda\left(\alpha+\beta\right) \leq Q$ and $t_k+\beta \rho^2=t_k\left(1+\beta \alpha^{-1}\right)=t_{k+1}$.
		
		Next we aim to prove $$
		Rc(g(t)) \geq \left(\epsilon  R(g(t))-\left(\frac{1}{4} \Gamma \sqrt{Q t_{k+1}}\right)^{-2}\right) g(t)
		$$
		on
		$B_{g_ 0}\left(p, r_{k+1}-\frac{1}{2} \Gamma \sqrt{Q t_{k+1}}\right) \times\left[0, t_{k+1}\right]$.
		
		For any $x_0\in B_{g_ 0}\left(p, r_{k+1}-\frac{1}{2} \Gamma \sqrt{Q t_{k+1}}\right) \times\left[0, t_{k+1}\right]$, we now consider the rescaled Yamabe flow $\tilde{g}(t)=\lambda_2^{-2} g\left(\lambda_2^2 t\right), t \in\left[0,16 \Gamma^{-2} Q^{-1}\right]$ where $\lambda_2=\frac{1}{4} \Gamma \sqrt{Q t_{k+1}}$ so that $B_{g_0}\left(x_0, \frac{1}{4} \Gamma \sqrt{Q t_{k+1}}\right)=B_{\tilde{g}(0)}(x_0, 1)$. Since $B_{g_0}\left(x_0, \frac{1}{4} \Gamma \sqrt{Q t_{k+1}}\right)\subset B_{g_ 0}\left(p, r_{k+1}-\frac{1}{2} \Gamma \sqrt{Q t_{k+1}}\right)\subset B_{g_ 0}\left(p, r_{k}-\frac{1}{2} \Gamma \sqrt{Q t_{k}}\right)$,
		we get from (\ref{key_pinching_1}) that
		$$
		Rc(\tilde{g}( t)) \geq \epsilon \left( R(\tilde{g}( t))-(1+\mu)^2\right)\tilde{g}(t)
		$$
		on $B_{\tilde{g}(0)}(x_0, 1)\times\left[0,16 (1+\mu)^{-2}\Gamma^{-2} Q^{-1}\right]$. By Lemma \ref{cut_off_1} and taking $r_0=\sqrt{t}$, we have
		\begin{equation}\label{3_1}
			\left(\frac{\partial}{\partial t}-(n-1)\Delta_{\tilde{g}(t)}\right) d_{\tilde{g}(t)}\left(x, x_0\right) \geq- \frac{2(n-1)(Q+1)}{\epsilon\sqrt{t}}-\frac{(1+\mu)^2}{\epsilon}d_{g(t_0)}\left(x, x_0\right)
		\end{equation}
		for any $(x,t)\in \left(B_{\tilde{g}(0)}(x_0, 1)\setminus B_{\tilde{g}(t)}(x_0, 2\sqrt{t})\right)\times\left[0,16 (1+\mu)^{-2}\Gamma^{-2} Q^{-1}\right]$. Moreover,  by (\ref{kand1}) we have
		$$
		| Rm(\tilde{g}(t))|  \leq 16^{-1} (1+\mu)^{2}\Gamma^{2} Q^{2}
		$$
		for $(x,t)\in B_{\tilde{g}(0)}(x_0, 1)\times\left[16 (1+\mu)^{-2}\Gamma^{-2} Q^{-1},16 \Gamma^{-2} Q^{-1}\right]$. Then by
		Lemma \ref{cut_off_2},  we have
		\begin{equation}\label{3_2}
			\left(\frac{\partial}{\partial t}-(n-1)\Delta_{\tilde{g}(t)}\right) d_{\tilde{g}(t)}\left(x, x_0\right)\geq-\frac{(n-1)(Q+1)}{\sqrt{t}}-n(n-1)16^{-1}\Gamma^{2} Q^{2} (1+\mu)^{2}d_{g(t_0)}\left(x, x_0\right)
		\end{equation}
		in the barrier sense for any $(x,t)\in \left(B_{\tilde{g}(0)}(x_0, 1)\setminus B_{\tilde{g}(t)}(x_0, 2\sqrt{t})\right)\times\left[16 (1+\mu)^{-2}\Gamma^{-2} Q^{-1},16 \Gamma^{-2} Q^{-1}\right]$. If we choose $\Gamma$ such that $n(n-1)16^{-1}\Gamma^{2} Q^{2}\ge \frac{1}{\epsilon}$, we conclude from (\ref{3_1}) and (\ref{3_2}) that
		\begin{equation*}
			\left(\frac{\partial}{\partial t}-(n-1)\Delta_{\tilde{g}(t)}\right) d_{\tilde{g}(t)}\left(x, x_0\right) \geq- \frac{c_1}{\sqrt{t}}-c_1\Gamma^2d_{\tilde{g}(t)}\left(x, x_0\right)
		\end{equation*}
		in the barrier sense for any $(x,t)\in \left(B_{\tilde{g}(0)}(x_0, 1)\setminus B_{\tilde{g}(t)}(x_0, 2\sqrt{t})\right)\times\left[0,16 \Gamma^{-2} Q^{-1}\right]$, where $c_1$ is the positive constant depending only on $\epsilon$ and $n$. It follows that
		\begin{equation*}
			\left(\frac{\partial}{\partial t}-(n-1)\Delta_{\tilde{g}(t)}\right) \left(e^{c_1\Gamma^2t}d_{\tilde{g}(t)}\left(x, x_0\right)+b\sqrt{t}\right) \geq 0
		\end{equation*}
		in the barrier sense for any $(x,t)\in \left(B_{\tilde{g}(0)}(x_0, 1)\setminus B_{\tilde{g}(t)}(x_0, 2\sqrt{t})\right)\times\left[0,16 \Gamma^{-2} Q^{-1}\right]$, where $b=2c_1e^{16c_1Q^{-1}}$.
		
		 Now take a
		cut-off function $\phi:[0,\infty]\to [0,1]$  such that $\phi \equiv 1$ on $[0, \frac{1}{2}]$, $\phi \equiv 0$ on $[1, \infty)$ and
		$$
		\phi^{\prime} \leq 0, \quad (\phi^{\prime})^2\le 10 \phi,\quad\phi^{\prime\prime} \ge - 10 \phi.
		$$
		To construct $\phi$ we can take 	
		$$
		\phi(y)= \begin{cases}1, &  y \le \frac{1}{2}; \\
			1-8(y-\frac{1}{2})^2,	  & \frac{1}{2} \le  y \le \frac{3}{4} ;\\
			8(y-1)^2, & \frac{3}{4} \le  y \le  1; \\
			0, & y\ge 1 ,
		\end{cases}
		$$	
		and smooth it slightly.	
		Setting
		$$
		\psi(x,t)=e^{-\frac{5(n-1)}{2}t}\phi\left(\frac{e^{c_1\Gamma^2t}d_{\tilde{g}(t)}\left(x, x_0\right)+b\sqrt{t}}{2e^{16c_1Q^{-1}}}\right).
		$$
	If choose  $\Gamma$ such that $\frac{1}{8}+b\Gamma^{-1}Q^{-\frac{1}{2}}e^{-16c_1Q^{-1}}\le 1$ and $4(2+be^{-16c_1Q^{-1}})\Gamma^{-1}Q^{-\frac{1}{2}}\le 1$, we have $\text{supp}(\psi)\subset B_{\tilde{g}(t)}(x_0,\frac{1}{4})$ and $B_{\tilde{g}(t)}(x_0,2\sqrt{t})\subset \{\phi=1\}$ for $t\in \left[0,16 \Gamma^{-2} Q^{-1}\right]$.
		Then
		we have
		\begin{equation}
			\left(\frac{\partial}{\partial t}-(n-1)\Delta_{\tilde{g}(t)}\right)\psi(x,t)\le 0,
		\end{equation}
		in the barrier sense on $B_{\tilde{g}(0)}(x_0, 1)\times\left[0,16 \Gamma^{-2} Q^{-1}\right]$.
		
		 By (2.12) and (2.13) in \cite{chow},
		$$
		\partial_t\left(Rc(\tilde{g}(x, t)) - \epsilon  R(\tilde{g}(x, t))\tilde{g}(x, t)\right)=(n-1) \Delta\left(Rc(\tilde{g}(x, t)) - \epsilon R(\tilde{g}(x, t))\tilde{g}(x, t)\right)+\frac{1}{n-2} B_{i j}.
		$$
		The eigenvalues of
		tensor $B_{i j}$ in the above are
		$\mu_i=\sum\limits_{k, l \neq i, k>l}\left(\lambda_k-\lambda_l\right)^2+(n-2) \sum\limits_{k \neq i}\left(\lambda_k-\lambda_i\right) \lambda_i$, where $\lambda_i$ is the  eigenvalue of Ricci tensor $Rc(\tilde{g}(x, t))$. Now define
		$$
		f(x, t)=\inf \{s \geq 0:Rc(\tilde{g}(x, t)) - \epsilon  R(\tilde{g}(x, t))\tilde{g}(x, t)+s \tilde{g}(x, t)>0\}.
		$$
	Similar computations as (74)-(75) in \cite{CZ2000}, we get
	\begin{equation}\label{m_estimate_2}
		\begin{aligned}
		\left(\frac{\partial}{\partial t}-(n-1)\Delta_{\tilde{g}(t)}\right)f
				\le&-\sum\limits_{k, l \neq 1, k>l}\left(\lambda_k-\lambda_l\right)^2-(n-2) \sum\limits_{k \neq 1}\left(\lambda_k-\lambda_1\right) \lambda_1\\
			\le& -\sum\limits_{k, l \neq 1, k>l}\left(\lambda_k-\lambda_l\right)^2-(n-2)\left(\frac{1}{\epsilon}-n\right)\lambda_1^2-\frac{n-2}{\epsilon}\lambda_1f\\
		\le&-\frac{n-2}{\epsilon}\lambda_1f \le \frac{L}{t}f	.
		\end{aligned}
	\end{equation}
		for $\epsilon\le \frac{1}{n}$ in the barrier sense,
		where $\lambda_1$ is the smallest eigenvalue of Ricci tensor $Rc(\tilde{g}(x, t))$ and $L$ is a positive constant only depending on $Q$ and $n$.
		
		 We now adapt the idea of Theorem 1.1 in \cite{LT}.  We consider the following function
		$$
		G=-f\psi^m+\eta,
		$$
		where $\eta(t)$ be a smooth positive function in $t$ such that $G>0$ near $t=0$. If $G(x,t)<0$ for some point on $B_{\tilde{g}(0)}(x_0, 1)\times\left[0,16 \Gamma^{-2} Q^{-1}\right]$, then there exists $(x_0,t_0)\in B_{\tilde{g}(0)}(x_0, 1)\times\left[0,16 \Gamma^{-2} Q^{-1}\right]$ such that $G(x_0,t_0)=0$ and $G(x,t)\ge 0$ on $B_{\tilde{g}(0)}(x_0, 1)\times\left[0,t_0\right]$.
		
		For any $\delta>0$, there exists $C^2$ functions $\tilde{f}(x), \tilde{\psi}(x)$ near $x_0$ such that $\tilde{\psi}(x) \leq \psi\left(x, t_0\right), \tilde{\psi}\left(x_0\right)=\psi\left(x_0, t_0\right), \tilde{f}(x) \leq f\left(x, t_0\right)$ and $\tilde{f}\left(x_0\right)=f\left(x_0, t_0\right)$ satisfying
		$$
		\frac{\partial_{-}}{\partial t}\psi\left(x_0, t_0\right)-\Delta_{\tilde{g}(t)} \tilde{\psi}(x_0) \leq \delta
		$$
		and
		$$
		\frac{\partial \_}{\partial t} f\left(x_0, t_0\right)-\Delta_{\tilde{g}(t)} \tilde{f}\left(x_0\right)\leq -\frac{n-2}{\epsilon}\lambda_1f+\delta
		$$
		here for a function $f(x, t)$,
		$$
		\frac{\partial_{-}}{\partial t} f\left(x_0, t_0\right)=\liminf _{h \rightarrow 0^{+}} \frac{f\left(x_0, t_0\right)-f\left(x_0, t_0-h\right)}{h} .
		$$
		Denote
		$$
		\tilde{G}(x,t)=-\tilde{f}(x)\tilde{\psi}(x)^m+\eta(t).
		$$
		Noted that $\tilde{G}(x_0,t_0)=0$ and $\tilde{G}(x,t_0)\ge 0$ for $x$ near $x_0$. It follows that at $(x_0,t_0)$
		\begin{equation}\label{m_estimate_1}
		\begin{aligned}
			0\le \Delta G&= -\tilde{\psi}^m\Delta\tilde{f}-\tilde{f}\Delta\tilde{\psi}^m-2\langle \nabla\tilde{\psi}^m,\nabla \tilde{f}\rangle\\
			&\le \psi^m\left(-\frac{\partial \_}{\partial t} f-\frac{n-2}{\epsilon}\lambda_1f+\delta\right)+mf\psi^{m-1}\left(-\frac{\partial \_}{\partial t} \psi+\delta\right)-2\langle \nabla\tilde{\psi}^m,\nabla \tilde{f}\rangle\\
			&= \frac{\partial \_}{\partial t} G-\eta^{\prime}-\frac{n-2}{\epsilon}\lambda_1f\psi^m+\delta\psi^m+m\delta f\psi^{m-1}-2\langle \nabla\tilde{\psi}^m,\nabla \tilde{f}\rangle\\
			&\le -\eta^{\prime}-\frac{n-2}{\epsilon}\lambda_1f\psi^m+\delta\psi^m+m\delta f\psi^{m-1}+2m^2\eta^{1-\frac{1}{m}}f^{\frac{1}{m}}\frac{|\nabla \psi|^2}{\psi}\\
			&\le -\eta^{\prime}+L\eta t^{-1}+\delta\psi^m+m\delta f\psi^{m-1}+2m^2c_2c_3\eta^{1-\frac{1}{m}}t^{-\frac{1}{m}},
		\end{aligned}			
		\end{equation}
		where we used
		$\nabla \tilde{f}=\frac{-m\tilde{f}\nabla \tilde{\psi}}{\tilde{\psi}}$,
		$|\nabla \tilde{\psi}|\le |\nabla \psi|$, $\eta=f\psi^m$ at $(x_0,t_0)$, $\frac{|\nabla \psi|^2}{\psi}\le c_2$ and $f^{\frac{1}{m}}\le c_3t^{-\frac{1}{m}}$ where $c_2$ is constant depending only on $n$ and  $c_3$ is constant depending only on $Q$ and $n$.

We will eventually take $\eta(t)=t^l$ for any $l\ge \alpha+1$, hence we first need to show
to show $G$ is positive
for $t$ near $0$.	By the second line in (\ref{m_estimate_2}), we get $	\left(\frac{\partial}{\partial t}-(n-1)\Delta_{\tilde{g}(t)}\right)f
\le K_1 f$, where $ K_1=\frac{n-2}{\epsilon}\sup \lambda_1$ near $t=0$. Then  it follows from  this and the fourth line in (\ref{m_estimate_1}),
by letting $\delta \rightarrow 0$, we conclude that at $(x_0, t_0)$,
		\begin{equation}\label{m_estimate_3}
	\begin{aligned}
	\eta^{\prime}\left(t_0\right) & \leq \eta\left(t_0\right)\left( K_1+c_4 m^2\left(\frac{K_2}{\eta\left(t_0\right)}\right)^{\frac{2}{m}}\right),  
		\end{aligned}			
\end{equation}
where $K_2=\sup f$ near $t=0$ and $c_4$ is a positive constant depending on $n$ and $Q$. In the above, we have used that at $\left(x_0, t_0\right)$,
$
\frac{1}{\psi^m}=\frac{f}{\eta} \leq \min \left\{\frac{\alpha}{t_0 \eta\left(t_0\right)}, \frac{a_0}{\eta\left(t_0\right)}\right\} .
$
First, we show that $f(t)=O\left(t^{1 / 2}\right)$. For any $1>\delta>0$, let $\eta(t)=$ $t^{\frac{1}{2}}+\delta$. We get from (\ref{m_estimate_3}) that
$$
\frac{1}{2} t_0^{-\frac{1}{2}} \leq\left(t_0^{\frac{1}{2}}+\delta\right)\left(K_1+\frac{c_4 m^2 K_2^{\frac{2}{m}}}{\left(t_0^{\frac{1}{2}}+\delta\right)^{\frac{2}{m}}}\right) .
$$	
Choose $m=2$, we see that there is $\tau>0$, which is independent of $\delta$, such that $t_0 \geq \tau$. Hence by letting $\delta \rightarrow 0$, we conclude that $f\psi^m \leq 2 t^{\frac{1}{2}}$ near $t=0$.
Next we improve the estimate.  Let $\eta=\delta t^{\frac{1}{4}}+t^k$ for any integer $k \geq 1$ and $\delta>0$. By (\ref{m_estimate_3}) again, we have
$$
\frac{1}{4} \delta t_0^{-\frac{1}{4}}+k t_0^{k-1} \leq\left(\delta t_0^{\frac{1}{4}}+t_0^k\right)\left(K_1+\frac{c_4 m^2 K_2^{\frac{2}{m}}}{\left(\delta t_0^{\frac{1}{4}}+t_0^k\right)^{\frac{2}{m}}}\right) .
$$
Choose $m$ large enough so that $2 k / m<1$, then we can find $\tau_1>0$ such that $t_0>\tau_1$. Therefore, we may conclude that $ f\psi^m  \leq 2 t^k$ near $t=0$.

		Then by letting $\delta\to 0$ in (\ref{m_estimate_1}) we have
		$$
		\eta^{\prime}\le L\eta t^{-1}+2m^2c_2c_3\eta^{1-\frac{1}{m}}t^{-\frac{1}{m}},
		$$
		at $(x_0,t_0)$. Taking $\eta(t)=t^l$, we get
		that
		$$
		lt_0^{l-1}\le Lt_0^{l-1}+2m^2c_2c_3t_0^{l-\frac{l+1}{m}}.
		$$
		Then for $l\ge L+1$, this implies
			that
		$$
		t_0^{l-1}\le 2m^2c_2c_3t_0^{l-\frac{l+1}{m}}.
		$$
		Choose $m$ suffient large such that $\frac{l+1}{m}<1$, we get that  $t_0\ge (2m^2c_2c_3)^{-\frac{m}{m-l-1}}$ and hence
		$$f\psi^m (x_0,t)\le t^l$$
			for $t\in \left[0,t_0\right]$.
		If choose $\Gamma$ such that $\left(16 \Gamma^{-2} Q^{-1}\right)^le^{40m(n-1)\Gamma^{-2}Q^{-1}}\le 1$ and $16\Gamma^{-2}Q^{-1}\le (2m^2c_2c_3)^{-\frac{m}{m-l-1}}$,
		we have
		$$f(x_0,t)\le 1$$
		for $t\in \left[0,16 \Gamma^{-2} Q^{-1}\right]$.
		Rescaling back to $g(t)$, we see that on $B_{g_ 0}\left(p, r_{k+1}+\frac{1}{2} \Gamma \sqrt{Q t_k}\right) \times\left[0, t_k\right]$, we have
		$$
		Rc(g(t)) \geq \epsilon  R(g(t))-\left(\frac{1}{4} \Gamma \sqrt{Q t_k}\right)^{-2},
		$$
		this proves $\mathcal{P}(k+1)$ holds if we choose that $\Gamma$ is larger than a constant depending only $\epsilon$, $n$.

		Since $ r_j$ is monotonely decreasing, we may assume there is $i \in \mathbb{N}$ such that $r_i \geq r+1$ and $r_{i+1}<r+1$. In particular, $\mathcal{P}(i)$ is true since $r_i>0$. We now estimate $t_i$.
		$$
		\begin{aligned}
			r+1>r_{i+1} & =r_1-\Gamma \sqrt{Q} \cdot \sum_{k=1}^i \sqrt{t_k} \\
			& \geq r+3-\Gamma \sqrt{Q t_i} \cdot \sum_{k=0}^{\infty}(1+\mu)^{-k} \\
			& =r+3-\sqrt{t_i} \cdot \frac{\Gamma \sqrt{Q}(1+\mu)}{\mu} .
		\end{aligned}
		$$
		This implies
		$$
		t_i>\frac{4 \mu^2}{Q \Gamma^2(1+\mu)^2}=: T(\epsilon) .
		$$
		In other words, there exists a smooth Yamabe flow solution $g(t)$ defined on $B_{g_0}(p, r+1) \times[0, T]$ so that $g(0)=g_0$ and $| Rm(g(t))| \leq \frac{Q}{t}$.
	\end{proof}

	\section{proof of theorem \ref{main} and theorem \ref{main_key}}\label{proof}
	Now we can give the proof of Theorem \ref{main} and theorem \ref{main_key}.
	
	\textbf{Proof of Theorem \ref{main} and theorem \ref{main_key}:}
We may assume that $\epsilon \in$ $\left(0, \frac{1}{n}\right)$ without loss of generality. Let $r_i \rightarrow+\infty$ and denote $h_{i, 0}=r_i^{-2} g_0$ so that $Rc\left(h_{i, 0}\right) \geq \epsilon \mathrm{R}\left(h_{i, 0}\right)$ on $M$. By Theorem \ref{main1}, there is a locally conformally flat Yamabe flow solution $h_i(t)$ on $B_{h_{i, 0}}(p, 1) \times[0, T]$ with
 $\left| Rm\left(h_i(t)\right)\right| \leq \frac{Q}{t}$
and $Rc\left(h_i(t)\right) \geq \left(\epsilon  R\left(h_i(t)\right)-1\right)h_i(t)$.
Define $g_i(t)=r_i^2 h_i\left(r_i^{-2} t\right)$ which is a Yamabe  flow solution on $B_{g_0}\left(p, r_i\right) \times$ $\left[0, T r_i^2\right]$ with
$$
\left\{\begin{array}{l}
	g_i(0)=g_0 ; \\
	\left| Rm\left(g_i(t)\right)\right| \leq \frac{Q}{t} \\
	Rc\left(g_i(t)\right) \geq \epsilon  R\left(g_i(t)\right)-r_i^{-2}
\end{array}\right.
$$
on  $B_{g_0}\left(p, r_i\right) \times\left(0, T r_i^2\right]$.
By Remark \ref{shi_esimates}, $g_i(t)$ subconverges to a smooth solution of the Yamabe flow on $M \times[0,+\infty)$ such that $g(0)=g_0,| Rm(g(t))| \leq \frac{Q}{t} $ and
$$
Rc(g(t)) \geq \epsilon  R(g(t)) g(t)
$$
for all $ t\in (0,+\infty)$.  By tracing this pinching estimate, we deduce that $ R \geq 0$.  And by applying strong maximum principle to (\ref{scalar}), we get $ R > 0$.

Actually, we now can conclude that Theorem \ref{main} holds by using Theorem 1.1 in \cite{MC} by Ma and the author.
However, since we have already proved the solution of Yamabe flow in our case is of Type III, i.e. $\sup\limits_{M\times [0,\infty)}t| Rm(g(t))| \leq Q $, we now can simplify
the proof in \cite{MC}. We next briefly introduce the proofs here for sake of convenience for the readers.

Notice that in advantange of Ricci flow, only assuming the Ricci curvature is nonnegative, Chow \cite{chow} proved that the following Harnack inequality holds for any smooth vector field $X$
 $$
 Z=\frac{\partial R}{\partial t}+\langle\nabla R, X\rangle+\frac{1}{2(n-1)} R_{i j} X^i X^j+\frac{R}{t} \geq 0
 $$
 for the  locally conformally flat manifolds Yamabe flow.
Taking $X=0$, we get $\frac{\partial}{\partial t}(t R) \geq 0$. Hence we have $A=\limsup\limits_{t \rightarrow \infty} t M(t)>0$, where $ M(t)=\sup\limits_{M} R(\cdot,t)$. Then we can take a sequence $\left(x_i, t_i\right)$ such that $t_i \rightarrow \infty$ and $A_i \doteq t_i R\left(g(x_i, t_i)\right) \rightarrow A$. Define the pointed rescaled solutions $\left(M^n, g_i(t), x_i\right), t \in\left(-t_i Q_i, \infty\right)$, by $g_i(t)=Q_i g\left(t_i+Q_i^{-1} t\right)$, where $Q_i=R\left(g(x_i, t_i)\right)$. For any $\epsilon_1>0$ we can find a time $\tau<\infty$ such that for $t \geq \tau$ and any $x \in M^n$,
 $
 t R(x, t) \leq A+\epsilon_1.
 $
 Then we have
 $
 \sup\limits_{M^n \times\left[-\frac{A_i\left(t_i-\tau\right)}{t_i}, \infty\right)}\left| R(g_i(t))\right|  \leq \frac{A+\epsilon_1}{A_i+t}
 $
  and
 $
 R(g_i(x_i, 0))=1.
 $
 Since the Weyl tensor of $M^n$ is vanishing and Ricci curvature is nonnegative, we get
 $$
 \sup _{M^n \times\left[-\frac{A_i\left(t_i-\tau\right)}{t_i}, \infty\right)}\left| Rm(g_i(t))\right| \leq \frac{A+\epsilon_1}{A_i+t} .
 $$
It follows that
we have a universal $\rho>0$ so that the conjugate radius on $(M^n,g_i(t))$ for $t\in \left[-\frac{A_i\left(t_i-\tau\right)}{t_i}, \infty\right)$ is always larger than $\rho$.
Therefore we can lift $\left(B_{g_i(0)}\left(x_i, \rho\right), g_i(t)\right)$ to $\left(B_{\text {euc }}(o_{x_i},\rho),\phi_{x_i}^* g_i(t)\right)$ by the exponential map of $g_i(0)$ at $x_i$, where $\phi_{x_i}=exp_{x_i,g_i(0)}$.  Then we deduce that $\left(B_{\text {euc }}(o_{x_i},\rho),\phi_{x_i}^* g_i(t)\right)$
subconverges to $(B_{\text {euc }}(o_{x_{\infty}},\rho), g_{\infty}(t) )$ in  $C^{\infty}$-sense.
Hence
$$
R(g_{\infty}(x, t)) \leq \frac{A}{A+t}
$$
for all $(x, t) \in B_{\text {euc }}(o_{x_{\infty}},\rho) \times(-A, \infty)$ and
$$
R(g_{\infty}(o_{x_{\infty}}, 0))=1 .
$$
So $R_{g_{\infty}}$ achieves at its maximum
at $(o_{x_{\infty}}, 0)$. It follows that $(B_{\text {euc }}(o_{x_{\infty}},\rho), g_{\infty}(t) )$ is a gradient expanding Yamabe soliton by the strong maximum principle; see Theorem 4.3 in \cite{MC}. Then we can use a technique used by A. Chau and L.F. Tam \cite{CT}  (see
the proof of Theorem 1.1 in \cite{MC} or the proof of Theorem 2.1 in \cite{CT} ) implies that the injectivity radius of $x_i$ has the uniformly lower bound
with respect to $g_i(0)$. With the injectivity bound, we  conclude that $\left(M^n, g_i(t), x_i\right)$ subconverges to a smooth limit with its universal covering is a complete noncompact gradient expanding Yamabe soliton satisfying $Rc(g_{\infty}(t)) \geq \epsilon  R(g_{\infty}(t)) g_{\infty}(t)>0$, which is a contradiction; see Theorem 4.4 in \cite{MC}.
	$\Box$

\section{Appendix}
In this section we give the proof of the maximum principle which we used in the proof of Lemma \ref{key_lemma}.
The lemma says that for a sequence of curvature flows on the balls  $ B_{\tilde{g}_k(0)}\left(\tilde{x}_k, r_k\right)\Subset M_k$  (possibly incomplete) satisfying
 all derivatives of curvatures are bounded and  the radii $r_k\to \infty$, but no uniform lower bound on the injectivity radius of $\tilde{g}_k(0)$,  we can apply the
 maximum principle on the smooth limit of the flows which obtained by lifting the flows via the exponential maps of $\tilde{g}_k(0)$.

\begin{thm}\label{maximun_principle}
 Suppose, for $t \in$ $[0, T]$ and $0<T<\infty$, that
	 $\{ B_{\tilde{g}_k(0)}\left(\tilde{x}_k, r_k\right), \tilde{g}_k(t),\tilde{x}_k
	\}_{k=1}^{\infty}$ is a sequence of flows with $r_k\to \infty$ for which $B_{\tilde{g}_k(0)}\left(\tilde{x}_k, r_k\right)\Subset M_k$  (possibly incomplete) and satisfy
	\begin{equation}\label{curvature_condition}
	\sup\limits_{B_{\tilde{g}_k(0)}\left(\tilde{x}_k, r_k\right)\times[0, T]}	|\nabla_{\tilde{g}(0)}^p \frac{\partial^q}{\partial t^q}\tilde{g}_k(t)|_{\tilde{g}(0)}\le C_{p,q}
	\end{equation}
	for all $p,q\ge 0$. Denote  $\rho>0$ be the a universal constant such that conjugate radius of $\left(	B_{\tilde{g}_k(0)}\left(\tilde{x}_k, r_k-\rho\right),\tilde{g}_k(t)\right)$ is always larger than $\rho$.
	 For any $r<\infty$ and any $y_k\in  B_{\tilde{g}_k(0)}\left(\tilde{x}_k, r\right)$, we can lift $\left(B_{\tilde{g}_k(0)}\left(y_k, \rho\right), \tilde{g}_k(t)\right)$ to $\left(B_{\text {euc }}(o_{y_k},\rho),\phi_k^{*}\tilde{g}_k(t)\right)$ by the exponential map of $\tilde{g}_k(0)$ at $y_k$ such that $\left(B_{\text {euc }}(o_{y_k},\rho),\phi_k^{*}\tilde{g}_k(t)\right)$ subconverges to $\left(B_{euc}\left(o_{y_{\infty}}, \rho\right), \tilde{g}_{\infty}(t)\right)$ in $C^{\infty}$ sense, where $\phi_k=exp_{y_k,\tilde{g}_k(0)}$
	  and $B_{\text {euc }}(o_{y_k},\rho)$ is the ball in the tangent space at $y_k$.
	Moreover, there exist a sequence of functions $f_{k} \in C^{\infty}(B_{\tilde{g}_k(0)}\left(\tilde{x}_k, r_k\right)\times[0, T], \mathbb{R})$
  satisfy
  \begin{equation}\label{f_bound}
\sup\limits_{B_{\tilde{g}_k(0)}\left(\tilde{x}_k, r_k\right)\times[0, T]}	|\nabla_{\tilde{g}(0)}^p \frac{\partial^q}{\partial t^q}f_k|_{\tilde{g}(0)}\le C_{p,q}
  \end{equation}
  for all $p,q\ge 0$ and $f_k\circ\phi_k$ subconverges to a smooth function $f^{y_{\infty}}_{\infty}$ in $C^{\infty}$-sense on
  $B_{\text {euc }}(o_{y_{\infty}},\rho)\times [0,T]$
  which solves
  \begin{equation}\label{f_equation}
\frac{\partial f^{y_{\infty}}_{\infty}}{\partial t} \leq \Delta f^{y_{\infty}}_{\infty}+\langle X_{y_{\infty}}, \nabla f^{y_{\infty}}_{\infty}\rangle+G(f^{y_{\infty}}_{\infty}, t) ,
  \end{equation}
where $G: \mathbb{R} \times[0, T] \rightarrow \mathbb{R}$ is Lipschitz and $X_{y_{\infty}}(\cdot,t)$ is smooth vector field such that $\sup\limits_{B_{\text {euc }}(o_{y_{\infty}},\rho)\times[0, T]}|X_{y_{\infty}}|$ is bounded by a constant which is independent of $y_{\infty}$ .
	Suppose further that $\psi:[0, T] \rightarrow \mathbb{R}$ solves
	$$
	\begin{cases}\frac{d \psi}{d t} & =G(\psi(t), t) \\ \psi(0) & =\alpha \in \mathbb{R} .\end{cases}
	$$
	If $f_{k}(\cdot, 0) \leq \alpha$, then $f_{\infty}(\cdot, t) \leq \psi(t)$ for all $t \in[0, T]$.
\end{thm}

\begin{proof}
By (\ref{curvature_condition}), we have a subsequence of $\{ B_{\tilde{g}_k(0)}\left(\tilde{x}_k,r_k\right), \tilde{g}_k(t),\tilde{x}_k
\}_{k=1}^{\infty}$, we still denote it by $\{ B_{\tilde{g}_k(0)}\left(\tilde{x}_k,r_k\right), \tilde{g}_k(t),\tilde{x}_k
\}_{k=1}^{\infty}$ for simplity,
converges to a complete metric space $(X_{\infty},d_{\infty} (t) , \tilde{x}_{\infty})$ for each $t\in (-\infty, 0]$ in pointed Gromov-Hausdorff distance.  By a standard diagonal argument, we can take points $\{y^m_{\infty}\}^{\infty}_{m=1}\subset X_{\infty}$
such that $\exists y^m_k\in B_{\tilde{g}_k(0)}\left(\tilde{x}_k, r_k\right)$,  $y^1_k=y_k$, $y^m_k\to y^m_{\infty}$ as $k\to \infty$,
 $\left(B_{\text {euc }}(o_{y^m_k},\rho),\phi_{y^m_k}^* \tilde{g}_k(t)\right)$
 converges to $\left(B_{\text {euc }}(o_{y^m_{\infty}},\rho), \tilde{g}_{y^m_{\infty}}(t)\right)$ in $C^{\infty}$-sense for any $m$ and
  $\mathop{\cup}\limits^{\infty}_{m=1}B_{d_{\infty}(t)}(y^m_{\infty},\frac{\rho}{2})=X_{\infty}$,  where $\phi_{y^m_k}=exp_{y^m_k,\tilde{g}_k(0)}$
  and $B_{\text {euc }}(o_{y^m_k},\rho)$ is the ball in the tangent space at $y^m_k$.	

For $\eta>0$, we consider the following ODE  $\phi_{\eta}:[0, T] \rightarrow \mathbb{R}$ solves
\begin{equation}\label{ODE}
	\begin{aligned}
	\begin{cases}\frac{d \psi_{\eta}}{d t} & =G(\psi_{\eta}(t), t)+\eta, \\ \psi_{\eta}(0) & =\alpha+\eta .\end{cases}
	\end{aligned}
\end{equation}
We only
need  prove that $f^{y_{\infty}^m}_{\infty}(\cdot, t) <\psi_{\eta}(t)$ for all $m$, $t \in [0, T]$ and arbitrary $\eta \in (0, \eta_0)$ for some $\eta_0$.
Let $F_k=f_k- \psi_{\eta}$. So $F_k\circ \phi_k$ subconverges to $F^{y_{\infty}^m}_{\infty}$ in $C^{\infty}$-sense on
$B_{\text {euc }}(o_{y^{m}_{\infty}},\rho)\times [0,T]$ for any $m$ with $F^{y_{\infty}^m}_{\infty}=f^{y_{\infty}^m}_{\infty}-\psi_{\eta}$. Let $S(t)=\sup\limits_{\substack{  B_{\text {euc }}(o_{y^m_{\infty}},\rho)\times [0,t]\\m\in \mathbb{N}}} F^{y^m_{\infty}}_{\infty}$. Since $S(0)< 0$, we argue by the contradition and assume $T_0\le T$ be the first time such that $S(T_0)=0$.

 	We will prove Theorem \ref{maximun_principle} by using the method which is in the spirit of Omori-Yau maximum principle.
Now we claim that there exist the sequences
 $y^j_{\infty}$ and $(z^j_{\infty},t_j)\in B_{\text {euc }}(o_{y^j_{\infty}},\rho)\times [0,T_0]$ such that
$$
|\nabla F^{y^j_{\infty}}_{\infty}|(z^j_{\infty},t_j)\le 2\epsilon_j, \frac{\partial}{\partial t} F^{y^j_{\infty}}_{\infty}(z^j_{\infty},t_j)\ge -  2C \epsilon_j, \Delta F^{y^j_{\infty}}_{\infty}(z^j_{\infty},t_j)\le 2C\epsilon_j,
$$
and
$$
F^{y^j_{\infty}}_{\infty}(z^j_{\infty},t_j)\to S(T_0)=0,
$$
where $C$ is positive constant depending only on $C_{p,q}$ and $n$.

By (\ref{f_bound}), we let $M=\sup\limits_{\substack{B_{\tilde{g}_k(0)}\left(\tilde{x}_k, r_k\right)\times[0, T_0]\\ k\in \mathbb{N}}}	|F_k|$.
For $\epsilon_j=\frac{1}{j}$ and $R_j=e^{jM}-1$, there exists $k_j>0$ such that when $k\ge k_j$ we have
\begin{equation}\label{g_estimate}
	\sup\limits_{B_{\text {euc }}(o_{y^m_{\infty}},\rho)\times [0,T_0]}|\partial^p( \phi_{y^m_k}^{*} \tilde{g}_k- \tilde{g}_{y^m_{\infty}})|\le \epsilon_j,
\end{equation}
for $ y_k^{m}\in B_{\tilde{g}_k(t)}\left(\tilde{x}_k, R_j-\rho\right)$ and
$B_{\tilde{g}_k(t)}(\tilde{x}_k,R_j)\subset\mathop{\cup}\limits^{\infty}_{m=1}B_{\tilde{g}_k(t)}(y^m_{k},\frac{\rho}{2})$.
Now we consider
$$
\bar{F}_{k_j}(x,t)=F_{k_j}(x,t)-\epsilon_j\log(r_{k_j}(x,t)+1),
$$
with $r_{k_j}(x,t)=d_{\tilde{g}_{k_j}}(x,t)$. Clearly, $\bar{F}_{k_j}(x,t)\le 0$ on $\left( B_{\tilde{g}_{k_j}(t)} \left(\tilde{x}_{k_j}, r_{k_j}\right)\backslash B_{\tilde{g}_{k_j}(t)}\left(\tilde{x}_{k_j}, R_j\right)\right)\times [0,T_0]$. So $\bar{F}_{k_j}$ attains its maximum at some point $(z_j,t_j)\in B_{\tilde{g}_{k_j}(t)}\left(\tilde{x}_{k_j}, R_j\right)\times (0,T_0]$ . Then we have
$$
|\nabla F_{k_j}|(z_j,t_j)\le \epsilon_j\frac{1}{r_{k_j}(z_j,t_j)+1}\le \epsilon_j,
$$
$$
\frac{\partial}{\partial t} F_{k_j}(z_j,t_j)\ge -\epsilon_j\frac{\frac{\partial}{\partial t} r_{k_j}(z_j,t_j) }{r_{k_j}(z_j,t_j)+1}\ge -C \epsilon_j,
$$
by (\ref{curvature_condition}) and
$$
\Delta F_{k_j}(z_j,t_j)\le \epsilon_j\frac{C}{r_{k_j}(z_j,t_j)+1}\le C\epsilon_j.
$$
by (\ref{curvature_condition}) and Laplacian comparison, where $C$ is positive constant depending only on $C_{p,q}$ and $n$.
It follows by (\ref{g_estimate}) that there exist $y^j_{\infty}$ and $(z^j_{\infty},t_j)\in B_{\text {euc }}(o_{y^j_{\infty}},\rho)\times [0,T_0]$ such that
$$
|\nabla F^{y^j_{\infty}}_{\infty}|(z^j_{\infty},t_j)\le 2\epsilon_j, \frac{\partial}{\partial t} F^{y^j_{\infty}}_{\infty}(z^j_{\infty},t_j)\ge -  2C \epsilon_j
$$
and
$$
\Delta F^{y^j_{\infty}}_{\infty}(z^j_{\infty},t_j)\le 2C\epsilon_j,
$$

 Next we show that $F^{y^j_{\infty}}_{\infty}(z^j_{\infty},t_j)\to S(T_0)=0$. Otherwise, we can take $(\hat{z}_{\infty},t_0)\in B_{\text {euc }}(o_{y^{j_0}_{\infty}},\rho)\times [0,T_0]$ for some $j_0$ and  $\delta>0$ such that
\begin{equation*}
	F^{y^{j_0}_{\infty}}_{\infty}(\hat{z}_{\infty},t_0) \ge F^{y^j_{\infty}}_{\infty}(z^j_{\infty},t_j)+\delta.
\end{equation*}
Denote $L=d_{\infty}(\tilde{x}_{\infty},y^{j_0}_{\infty})$.
Then by (\ref{g_estimate}) we conclude that
there exists $(\hat{z}_{j},t_0)\in B_{\tilde{g}_k(0)}\left(\tilde{x}_{k_j}, R_j\right)\times [0,T]$ such that
\begin{equation}\label{e_1}
	F_{k_j}(\hat{z}_{j},t_0) \ge F_{k_j}(z_{j},t_j)+\delta-C\epsilon_j,
\end{equation}
with $r_{k_j}(\hat{z}_{j}, t_0)\le L+\rho+1$ when $j$ sufficient large. By the definition of $z_j$, we have
$$
F_{k_j}(z_{j},t_j)-\epsilon_j\log(r_{k_j}(z_{j},t_j)+1)=\bar{F}_{k_j}(z_{j},t_j)\ge \bar{F}_{k_j}(\hat{z}_{j},t_0)=F_{k_j}(\hat{z}_{j},t_0)-\epsilon_j\log(r_{k_j}(\hat{z}_{j},t_0)+1).
$$
Hence
$$
F_{k_j}(z_{j},t_j)\ge F_{k_j}(\hat{z}_{j},t_0)-\epsilon_j(L+\rho+1),
$$
which contradicts to (\ref{e_1}) when $j$ sufficient large. This proves the claim. 
Since $F^{y^j_{\infty}}_{\infty}$ satifying
 \begin{equation*}
	\frac{\partial F^{y^j_{\infty}}_{\infty}}{\partial t} \leq \Delta F^{y^j_{\infty}}_{\infty}+\langle X_{y^j_{\infty}}, \nabla F^{y^j_{\infty}}_{\infty}\rangle+G(f^{y^j_{\infty}}_{\infty}, t_j)-G(\psi_{\eta}(t_j), t_j)-\eta 
\end{equation*}
at $(z^j_{\infty},t_j)$, we get the contradiction by the claim when
we take $j\to\infty$.

\end{proof}

\end{document}